\documentclass[12pt,centertags,oneside]{amsart}
\usepackage{amsmath,amstext,amsthm,amscd,typearea,hyperref}
\usepackage{amssymb}
\usepackage{a4wide}
\usepackage[mathscr]{eucal}
\usepackage{mathrsfs}
\usepackage{typearea}
\usepackage{charter}
\usepackage{pdfsync}
\usepackage[a4paper,width=16.5cm,top=3cm,bottom=3cm]{geometry}

\numberwithin{equation}{section}

\allowdisplaybreaks
\emergencystretch=\maxdimen
\usepackage{multicol}

\usepackage{xcolor}

\newtheorem{theorem}{Theorem}[section]
\newtheorem{definition}[theorem]{Definition}
\newtheorem{proposition}[theorem]{Proposition}
\newtheorem{corollary}[theorem]{Corollary}
\newtheorem{lemma}[theorem]{Lemma}

\newtheorem*{definition*}{Definition}

\newtheorem{mainthm}{Theorem}

\newcommand{\cali}[1]{\mathscr{#1}}

\newcommand{\GL}{{\rm GL}}

\newcommand{\supp}{{\rm supp}}

\newcommand{\diff}{{\rm d}}

\renewcommand{\Re}{\mathop{\mathrm{Re}}}

\renewcommand{\GL}{{\rm GL}}

\newcommand{\ep}{\epsilon}

\newcommand{\Cc}{\cali{C}}

\newcommand{\Tc}{\cali{T}}

\newcommand{\B}{\mathbb{B}}

\newcommand{\C}{\mathbb{C}}

\newcommand{\N}{\mathbb{N}}

\newcommand{\R}{\mathbb{R}}

\renewcommand\P{\mathbb{P}}

\newcommand{\E}{\mathbf{E}}

\newcommand{\lp}{\langle}
\newcommand{\rp}{\rangle}

\newcommand{\norm}[1]{\lVert#1\rVert}

\newcommand{\oA}{\mathcal{A}}
\newcommand{\oB}{\mathcal{B}}
\newcommand{\oF}{\mathcal{F}}
\newcommand{\oP}{\mathcal{P}}
\newcommand{\oN}{\mathcal{N}}
\newcommand{\oR}{\mathcal{R}}
\newcommand{\oE}{\mathcal{E}}

\newcommand{\oQ}{\mathcal{Q}}
\newcommand{\oS}{\mathcal{S}}
\newcommand{\oL}{\mathcal{L}}

\usepackage{fancyhdr}

\title{Berry-Esseen bound and Local Limit Theorem for the coefficients \\ of products of random matrices}

\author{Tien-Cuong Dinh}
\address{Department of Mathematics,  National University of Singapore - 10, Lower Kent Ridge Road - Singapore 119076}
\email{matdtc@nus.edu.sg}

\author{Lucas Kaufmann}
\address{Center for Complex Geometry - Institute for Basic Science (IBS) - 55 Expo-ro Yuseong-gu Daejeon 34126 South Korea}
\email{lucaskaufmann@ibs.re.kr}

\author{Hao Wu}
\address{Department of Mathematics,  National University of Singapore - 10, Lower Kent Ridge Road - Singapore 119076}
\email{matwu@nus.edu.sg}

\thanks{This work was supported by the NUS and MOE grants  AcRF Tier 1 R-146-000-319-114 and MOE-T2EP20120-0010. L. Kaufmann was supported by the Institute for Basic Science (IBS-R032-D1)}

\begin{document}

\begin{abstract}
Let $\mu$ be a probability measure on $\GL_d(\R)$ and denote by $S_n:= g_n \cdots g_1$ the associated random matrix product, where $g_j$ are i.i.d.\ with law $\mu$. Under the assumptions that $\mu$ has a finite exponential moment and generates a proximal and strongly irreducible semigroup, we prove a Berry-Esseen bound with the optimal rate $O(1/\sqrt n)$ and a general Local Limit Theorem for the coefficients of $S_n$.
\end{abstract}

\clearpage\maketitle
\thispagestyle{empty}


\noindent\textbf{Mathematics Subject Classification 2020:} \texttt{60B15,60B20,60F99,37A30}.


\section{Introduction}

Let $\mu$ be a probability on $G:=\GL_d(\R)$, $d \geq 2$.  Then, $\mu$ induces a random walk on $G$ by letting  $$S_n: = g_n \cdots g_1,$$ where  $n \geq 1$ and the $g_j$'s are independent and identically distributed random elements of $G$ with law given by $\mu$.  The study of these random processes and associated limit theorems has a rich history,  starting from seminal works of Furstenberg and Kesten \cite{furstenberg-kesten,furstenberg} leading to important progress since then.  This topic is still very active, with important new results and techniques being recently discovered.  We refer to \cite{bougerol-lacroix,benoist-quint:book} for an overview. See also below for some recent results.

We consider the standard linear action of $G$ on $\R^d$ and the induced action on the real projective space $\P^{d-1}$.  Denote by $\| v \|$ the standard euclidean norm of $v \in \R^d$ and,  for $g \in G$, let $\|g\|$ be the associated operator norm.

In order to study the random matrices $S_n$,   it is useful to look at associated real-valued random variables.  An important function in this setting  is the \textit{norm cocycle},  defined by $$\sigma(g,x) = \log \frac{\norm{gv}}{\norm{v}}, \quad \text{for }\,\, v \in \R^d \setminus \{0\}, \, x = [v] \in \P^{d-1}   \, \text{ and } g \in G.$$
The cocycle relation $\sigma(g_2g_1,x) = \sigma(g_2,g_1 \cdot x) + \sigma(g_1,x)$ can be used  to effectively apply methods such as the spectral theory of complex transfer operators (see Subection \ref{subsec:markov-op}) and martingale approximation \cite{benoist-quint:CLT}.  Some other significant quantities are: the norm $\|g\|$,  the spectral radius $\rho(g)$ and the coefficients of $g$, the latter being object of this article.

The goal of this work is to obtain two new limit theorems for the coefficients of $S_n$ as $n$ tends to infinity.  For $v \in \R^d$ and $f \in (\R^d)^*$,  its dual space,  we denote by $\lp f,v \rp := f(v)$ their natural coupling.  Observe that the $(i,j)$-entry of a matrix $g$ is given by $\lp e_i^* , g  e_j \rp$,  where $(e_k)_{1\leq k \leq d}$ (resp. $(e^*_k)_{1\leq k \leq d}$) denotes the canonical basis of $\R^d$ (resp.$(\R^d)^*$).   Our results will apply,  more generally,  to the random variables of the form
$$\log{ |\lp f, S_n v\rp | \over \norm{f} \norm{v}},$$
with  $v \in \R^d \setminus \{0\}$ and $f \in (\R^d)^* \setminus \{0\}$.

In order to obtain meaningful results,  some standard assumptions on the measure $\mu$ need to be made.   Recall that a matrix $g\in G$ is said to be \textit{proximal} if it admits a unique eigenvalue of maximal modulus which is moreover of multiplicity one.  Let  $\Gamma_\mu$ be the smallest closed semigroup containing the support of $\mu$.  We assume that $\Gamma_\mu$ is \textit{proximal}, that is,  it contains a proximal matrix,  and  \textit{strongly irreducible},  that is,  the action of $\Gamma_\mu$ on $\R^d$ does not preserve a finite union of proper linear subspaces.  It is well-known that, under the above conditions, $\mu$ admits a unique stationary probability measure on $\P^{d-1}$, see Section \ref{sec:prelim}. 

We'll also assume that $\mu$ has a \textit{finite exponential moment}, that is,  $\int_{G} N(g)^\varepsilon \diff\mu(g) < \infty$ for some $\varepsilon>0$, where $N(g):=\max\big( \norm{g},\norm{g^{-1}} \big)$. 

\medskip

Our first result is a Berry-Esseen bound with rate $O(1/ \sqrt n)$ for the coefficients,  which is a quantitative version of the Central Limit Theorem (CLT).   For the CLT for the coefficients without convergence rate,  see \cite{benoist-quint:book}.  The \textit{first Lyapunov exponent} of $\mu$ is,  by definition, the number
$$\gamma := \lim_{n \to \infty} \frac1n \int \log\|g_n \cdots g_1\| \, \diff \mu(g_1) \cdots \diff \mu(g_n).$$

\begin{mainthm}\label{thm:BE-coeff}
	Let $\mu$ be a  probability measure on $\GL_d(\R)$.  Assume that $\mu$ has a finite exponential moment and  that $\Gamma_\mu$ is proximal and strongly irreducible.  Let $\gamma$ be the associated first Lyapunov exponent. Then, there is a constant $C>0$ and a real number $\varrho > 0$,  such that, for any $x:=[v]\in \P^{d-1}, y:=[f]\in (\P^{d-1})^*$, any interval $J\subset\R$, and all $n\geq 1$, we have
	$$\bigg| \mathbf P \Big(   \log{ |\lp f, S_n v\rp | \over \norm{f} \norm{v}} - n \gamma\in \sqrt n J \Big)  -  \frac{1}{\sqrt{2 \pi} \, \varrho}  \int_{J} e^{-\frac{s^2}{2 \varrho^2}} \, \diff s \bigg| \leq \frac{C}{\sqrt n}.$$
\end{mainthm}

We observe that the rate $O(1/ \sqrt n)$ in the above theorem is optimal  as this is also the case for sums of real-valued i.i.d.'s.  Many related bounds for the other random variables associated with $S_n$ mentioned above can be found in the recent literature.  More details are given below.    

\medskip

Our second result is a Local Limit Theorem for the coefficients.

\begin{mainthm}\label{thm:LLT-coeff}
	Let $\mu$ be a  probability measure on $\GL_d(\R)$.  Assume that $\mu$ has a finite exponential moment and  that $\Gamma_\mu$ is proximal and strongly irreducible.  Let $\gamma$ be the associated first Lyapunov exponent.  Let $\varrho > 0$ be as in Theorem \ref{thm:BE-coeff}. Then,  for any $x:=[v]\in \P^{d-1}, y:=[f]\in (\P^{d-1})^*$ and any $-\infty<a<b<\infty$, we have
	$$\lim_{n\to \infty}\sup_{t\in\R}\bigg| \sqrt{n}  \, \mathbf P \Big(  t+ \log{ |\lp f, S_n v\rp | \over \norm{f} \norm{v}} - n \gamma\in [ a, b] \Big)  -   e^{-\frac{t^2}{2 \varrho^2 n}} {b-a\over \sqrt{2 \pi}\,\varrho} \bigg| =0.$$
Moreover, the convergence is uniform in $x \in \P^{d-1}$ and $y \in (\P^{d-1})^*$.
\end{mainthm}

As discussed below,  Theorem \ref{thm:LLT-coeff} contains a recent result of  Grama-Quint-Xiao \cite{grama-quint-xiao}.
\medskip

\noindent \textbf{Related works.} As mentioned before, the  rate $O(1  / \sqrt n)$ in Theorem \ref{thm:BE-coeff} is optimal.  Before our work, Berry-Esseen bounds for the coefficients of $S_n$ were only known under strong positivity conditions on the matrices in the support of $\mu$, see \cite{xiao-grama-liu:norm-coeff}.   Under the assumptions of Theorem \ref{thm:BE-coeff},  it is known for a long time that one can obtain a Berry-Esseen bound for the norm cocycle $\sigma(S_n,x)$  with rate $O(1  / \sqrt n)$,  see \cite{lepage:theoremes-limites,bougerol-lacroix} and \cite{fernando-pene} for a refined version.  For the variables $\log \|S_n \|$ and $\rho(S_n)$,  the progress is more recent and in these cases a Berry-Esseen bound with  rate $O(1  / \sqrt n)$ is  known under strong positivity conditions   and, without such conditions, a $O(\log n  / \sqrt n)$ rate can be obtained, see \cite{xiao-grama-liu:hal,xiao-grama-liu:norm-coeff}.

The exponential moment condition in Theorem \ref{thm:BE-coeff} is stronger than what one should require.  Parallel to the case of sums of i.i.d.'s,  one should expect to have the same result under a third moment condition, that is, $\int_G \big( \log N(g) \big)^3 \, \diff \mu(g) < + \infty$.  This is unknown for the coefficients. Under this condition,  for the norm cocycle  $\sigma(S_n,x)$,  the best known rate is $O(n^{-1 \slash 4} \sqrt{\log n})$ obtained in \cite{cuny-dedecker-jan} using martingale approximation methods in the spirit of \cite{benoist-quint:CLT}.  This has been recently improved in  \cite{cuny-dedecker-merlevede-peligrad}  to a  $O(1  / \sqrt n)$ (resp. $O((\log n)^{1/2}n^{-1 \slash 2})$) rate under a fourth (resp.  third) moment condition.  See also \cite{jirak:BE1} for related results under low moment conditions.    In the particular case where $d=2$, the authors have obtained the optimal $O(1  / \sqrt n)$ rate under a third moment condition  \cite{DKW:LLT} .

Concerning the Local Limit Theorem (LLT),   Theorem \ref{thm:LLT-coeff} above strengthens a recent result of Grama-Quint-Xiao \cite{grama-quint-xiao}, which holds under the same hypothesis as Theorem \ref{thm:LLT-coeff},  but only for the parameter $t = 0$.  See also \cite{xiao-grama-liu:coeff-large-deviation} for related results and \cite{DKW:LLT} for the case $d=2$ under a third moment condition.  These limit theorems allow us to estimate the probability that the random variables  $ \frac{1}{\sqrt n} \big(\log{ |\lp f, S_n v \rp | \over \norm{f} \norm{v}} - n \gamma \big)$ fall on intervals of size $O(1  / \sqrt n)$ around the origin, while Theorem \ref{thm:LLT-coeff} works for intervals of size $O(1 /  \sqrt n)$ around an arbitrary point on the real line.  For the norm cocycle,  the  general LLT is due to Le Page \cite{lepage:theoremes-limites}. 

\medskip

\noindent \textbf{Overview  of the proofs.} When proving limit theorems for the coefficients,  the first step is to compare them with the norm cocycle via the elementary identity
\begin{equation*}
\log{ |\lp f, S_n v \rp | \over \norm{f} \norm{v}} = \sigma(S_n,x) + \log \Delta(S_n x,y),
\end{equation*}
where $\Delta(x,y):= \frac{ |\lp f, v \rp |}{\norm{f} \norm{v}}$.  One can  check that $\Delta(x,y) = d (x,H_y)$, where $H_y := \P(\ker f)$ is a hyperplane in $\P^{d-1}$ and $d$ is a natural distance on $\P^{d-1}$ (see Section \ref{sec:prelim}).   Then, we can use the above formula and work with the random variable $\sigma(S_n,x) +\log d(S_n x, H_y)$ instead of $\log{ |\lp f, S_n v \rp | \over \norm{f} \norm{v}} $.  The behaviour of $\sigma(S_n,x)$ can be studied via the perturbed Markov operators (see Subsection \ref{subsec:markov-op}). The term  $\log d(S_n x, H_y)$ is handled using some large deviation estimates  combined with a good partition of unity (see Lemmas \ref{lemma:partition-of-unity} and \ref{lemma:partition-of-unity-2}).   The latter is one of our key arguments,  applied to approximate the quantity $\sigma(S_n,x) + \log d(S_n x,H_y)$ by a sum of functions of two separate variables $\sigma(S_n,x)$ and $S_n x$, see also \cite{grama-quint-xiao}. We use a partition of $\P^{d-1} \setminus H_y$ by functions $(\chi_k)_{k \geq 0}$  subordinated to ``annuli'' around $H_y$ of the form  $ \big\{ w \in \P^{d-1} :\, e^{-k-1} < d(w,H_y) < e^{-k+1} \big\}$.  This allows us to have a good control on the errors in a ``uniform'' manner,   which is responsible for the sharp bounds.  In particular,  we don't need to use the zero-one law for algebraic subsets of $\P^{d-1}$ obtained in \cite{grama-quint-xiao},  which is a main ingredient in the proof of their version of the LLT. 

For most of our  estimates,  we strongly  rely on the spectral analysis of the Markov operator and its pertubations on a H\"older space $\Cc^\alpha(\P^{d-1})$  (see Subsection \ref{subsec:markov-op}).  It is crucial to choose $\alpha$ small in order to reduce the impact of the norm of $\chi_k$ when $k$ is large, see Lemmas \ref{lemma:norm-Phi} and  \ref{lemma:norm-Phi-2}.  A main difficulty that appeared in our computations is how to to handle the ``tail" of the approximation using $\chi_k$.  To overcome this problem, we introduce  an auxiliary function 
$$\Phi_{n}^{\star}  (w):= 1 - \sum_{0\leq k\leq A\log n} \chi_k(w)$$ for some well-chosen $A>0$,  which has negligible impact on the estimates but whose presence is helpful in the computations,  see e.g.  Lemmas \ref{lemma:theta-1-bound},  \ref{lemma:theta-2-bound} and \ref{lemma-R-S}.

Our approach can also be applied to the case of more general target functions.  More precisely, we can replace the probabilities in Theorems \ref{thm:BE-coeff} and \ref{thm:LLT-coeff} by the expectation of some good test functions on $\R\times \P^{d-1}$. We postpone these questions to a future work  in order to keep the current article less technical.  The results presented here can be extended to the case of matrices with entries in a local field,  see \cite{benoist-quint:book} for local field versions of the results stated in Section \ref{sec:prelim}.

\medskip

\noindent \textbf{Organization of the article.} The article is organized as follows.  In Section \ref{sec:prelim}, we recall some standard result from the theory of random matrix products that will be used in the proofs, most notably: spectral gap results, large deviation estimates and regularity properties of the stationary measure.  Theorem \ref{thm:BE-coeff} is proved in Section \ref{sec:BE} and Theorem \ref{thm:LLT-coeff} is proved in Section \ref{sec:LLT}.

\medskip

\noindent\textbf{Notations.} Throughout this article, the symbols $\lesssim$ and $\gtrsim$ stand for inequalities up to a multiplicative constant.  The dependence of these constants on certain parameters (or lack thereof),  if not explicitly stated,  will be clear from the context.  We denote by $\mathbf E$ the expectation and $\mathbf P$ the probability.

\section{Preliminary results} \label{sec:prelim}

We start with some basic results and notations. We refer to \cite{bougerol-lacroix,benoist-quint:book} for the proofs of the results described here. See also \cite{lepage:theoremes-limites}.

\subsection{Norm cocycle,  first Lyapunov exponent and the stationary measure}
Let  $G:=\GL_d(\R)$. We consider its standard linear action on $\R^d$ and the induced action on the real projective space $\P^{d-1}$. Let $\mu$ be a probability measure on $G$.   For $n \geq 1$,  we define the convolution measure by  $\mu^{*n} := \mu * \cdots * \mu$ ($n$ times) as the push-forward of the product measure $\mu^{\otimes n}$ on $G^n$ by the map $(g_1, \ldots, g_n) \mapsto g_n \cdots g_1$.  If $g_j$ are i.i.d. random matrices with law $\mu$ then $\mu^{*n} $ is the law of $S_n := g_n \cdots g_1$.

Denote by $\norm{g}$ the operator norm of the matrix $g$ and define $N(g):=\max\big( \norm{g},\norm{g^{-1}} \big)$. We say that  $\mu$  has a \textit{finite exponential moment} if  
$$\E \big( N(g)^\varepsilon \big) = \int_G N(g)^\varepsilon \, \diff \mu(g) < \infty \quad \text{for some } \,\, \varepsilon > 0.$$

The  \textit{first Lyapunov exponent} is the number
$$\gamma := \lim_{n \to \infty} \frac1n \E \big( \log\|S_n\| \big)=\lim_{n \to \infty} \frac1n \int \log\|g_n \cdots g_1\| \, \diff \mu(g_1) \cdots \diff \mu(g_n).$$

The \textit{norm cocycle} is the function $\sigma: G \times \P^{d-1} \to \R$ given by $$\sigma(g,x) = \sigma_g(x):= \log \frac{\norm{gv}}{\norm{v}}, \quad \text{for }\,\, v \in \R^d \setminus \{0\}, \, x = [v] \in \P^{d-1}   \, \text{ and } g \in G.$$

An element $g\in G$ is said to be \textit{proximal} if it admits a unique eigenvalue of maximal modulus which is moreover of multiplicity one. A semigroup $\Gamma$ is said to be \textit{proximal} if it contains a proximal element. We say that (the action of) $\Gamma$ is \textit{strongly irreducible} if it does not preserve a finite union of proper linear subspaces of $\R^d$.

 Denote by $\Gamma_\mu$ the semigroup generated by the support of $\mu$. If $\Gamma_\mu$ is proximal and strongly irreducible, then $\mu$ admits a unique \textit{stationary measure}, that is, a probability measure $\nu$ on $\P^{d-1}$ satisfying $$\int_G g_* \nu \, \diff \mu(g)= \nu.$$
 The above measure is also called the \textit{Furstenberg measure} associated with $\mu$. 

\subsection{Large deviation estimates and regularity}

We equip $\P^{d-1}$ with a natural distance given by 
\begin{equation*} 
d(x,w) : = \sqrt{1 - \bigg( \frac{\langle v_x,v_w \rangle}{\|v_x\| \|v_w\|} \bigg)^2}, \quad \text{where} \quad v_x,v_w \in \R^d \setminus \{0\}, \, x = [v_x], \,\, w = [v_w] \in \P^{d-1}.
\end{equation*}

Observe that $d(x,w)$ is the sine of the angle between the lines $x$ and $w$ in $\R^d$. Then,  $(\P^{d-1}, d)$ has diameter one on which the orthogonal group $\text{O}(d)$ acts transitively and isometrically. We will denote by $\B(x,r)$ the associated open ball of center $x$ and radius $r$ in $\P^{d-1}$.
\medskip

For $y\in (\P^{d-1})^*$,  the dual of $\P^{d-1}$, we  denote by $H_y$ the kernel of $y$, which is a (projective) hyperplane in $\P^{d-1}$. We'll need the following large deviation estimates.  Recall  that $\gamma$ denotes the first Lyapunov exponent of $\mu$.

\begin{proposition}[\cite{benoist-quint:book}--Proposition 14.3 and Lemma 14.11]  \label{prop:BQLDT}
	Let $\mu$ be a  probability measure on $G=\GL_d(\R)$. Assume that $\mu$ has a finite exponential moment and that $\Gamma_\mu$ is proximal and strongly irreducible. Then, for any $\ep>0$ there exist $c>0$ and $n_0 \in\N$ such that, for all $\ell\geq n\geq n_0$, $x\in \P^{d-1}$ and $y\in(\P^{d-1})^*$, one has 
		$$   \mu^{*n} \big\{g\in G:\, |\sigma(g,x)-n\gamma|  \geq n\ep \big\}\leq e^{-cn}          $$
		and
		$$   \mu^{*\ell} \big\{g\in G:\, d(gx, H_y) \leq e^{-\ep n}  \big\}  \leq e^{-cn}.         $$
\end{proposition}

The next result gives a regularity property of the stationary measure $\nu$. See also \cite{benoist-quint:CLT,DKW:PAMQ} for the case where $\mu$ satisfies weaker moment conditions. For a hyperplane $H$ in $\P^{d-1}$ and $r>0$,  we denote $\B(H,r) :=\{x \in \P^{d-1}: d(x,H) < r\}$, which is a ``tubular'' neighborhood  of $H$.

\begin{proposition}[\cite{guivarch:1990}, \cite{benoist-quint:book}--Theorem 14.1]\label{prop:regularity}
	Let $\mu$ be a  probability measure on $G=\GL_d(\R)$. Assume that $\mu$ has a finite exponential moment and that $\Gamma_\mu$ is proximal and strongly irreducible. Let $\nu$ be the associated stationary measure. Then, there are constants $C>0$ and $\eta>0$ such that 
	$$\nu\big(\B(H_y,r)\big)\leq C r^\eta \quad\text{for every} \quad y\in (\P^{d-1})^* \, \, \text{ and } \,\, 0 \leq r \leq 1.$$
\end{proposition}

\subsection{The Markov operator and its perturbations} \label{subsec:markov-op}

The \textit{Markov operator} associated to $\mu$ is the operator 
$$\oP \varphi(x):=\int_{G} \varphi(gx) \,\diff\mu(g),$$
acting on functions on $\P^{d-1}$.

For $z\in\C$, we consider the perturbation $\oP_z$ of $\oP$ given by
$$\oP_z \varphi(x):=\int_{G} e^{z\sigma(g,x)}\varphi(gx) \,\diff\mu(g),$$
where $\sigma(g,x)$ is the norm cocycle defined above. The operator $\oP_z$ is often called the \textit{complex transfer operator}. Notice that $\oP_0= \oP$ is the original Markov operator. A direct computation using the cocycle relation $\sigma(g_2g_1,x) = \sigma(g_2,g_1  x) + \sigma(g_1,x)$ gives that
\begin{equation} \label{eq:markov-op-iterate}
\oP^n_z \varphi (x)   = \int_G e^{z \sigma(g,x)} \varphi(gx) \, \diff \mu^{* n} (g).
\end{equation}
In other words, $\oP^n_z$ corresponds to the perturbed Markov operator associated with the convolution power $\mu^{\ast n}$.

We recall some fundamental results of Le Page about the spectral properties of the above operators. For $0<\alpha<1$, we denote by $\Cc^\alpha(\P^{d-1})$ the space of H\"older continuous functions on $\P^{d-1}$ equipped with the norm
\begin{equation*}
\|\varphi\|_{\Cc^\alpha} := \|\varphi\|_\infty + \sup_{x \neq y \in \P^{d-1}} \frac{|\varphi(x)-\varphi(y)|}{d(x,y)^\alpha}. 
\end{equation*}

Recall that the essential spectrum of an operator is the subset of the spectrum obtained by removing its isolated points corresponding to eigenvalues of finite multiplicity.
The essential spectral radius $\rho_{\rm ess}$ is then the radius of the smallest disc centered at the origin which contains the essential spectrum.

\begin{theorem}\label{thm:spectral-gap} \cite{lepage:theoremes-limites}, \cite[V.2]{bougerol-lacroix} Let $\mu$ be a  probability measure on $G=\GL_d(\R)$ with a finite exponential moment such that $\Gamma_\mu$ is proximal and strongly irreducible. Then, there exists an $0<\alpha_0 <1$ such that, for all $0<\alpha \leq \alpha_0$, the operator $\oP$ acts continuously on $\Cc^\alpha(\P^{d-1})$ with a spectral gap. In other words,  $\rho_{\rm ess}(\oP)<1$ and $\oP$ has a single eigenvalue of modulus $\geq 1$ located at $1$, which is isolated and of multiplicity one.
\end{theorem}

It follows directly from the above theorem that $\|\oP^n - \oN\|_{\Cc^\alpha} \leq C \lambda^n$ for some constants $C > 0$ and $0<\lambda<1$, where $\oN$ is the projection $\varphi \mapsto \big( \int_{\P^{d-1}} \varphi \,  \diff \nu \big) \cdot \mathbf 1$ onto the space of constant functions. Here and in what follows, we denote by  $\mathbf 1$ the constant function equal to $1$ on $\P^{d-1}$.

The following result gives the regularity of the family of operators $z \mapsto \oP_z$. The second part follows from the general theory of perturbations of linear operators, which implies that the spectral properties of $\oP_0$ persist for small values of $z$. For a proof,  see e.g.\ \cite[V.4]{bougerol-lacroix}.

\begin{proposition} \label{prop:spectral-decomp}
 Let $\mu$ and $\alpha_0$ be as in Theorem \ref{thm:spectral-gap}. There exists $b > 0$ such that for $|\Re z| < b$, the operators $\oP_z$ act continuously on $\Cc^\alpha(\P^{d-1})$ for all $0<\alpha \leq \alpha_0$. Moreover, the family of operators $z \mapsto \oP_z$ is analytic near $z=0$.

  In particular, there exists an $\epsilon_0 > 0$ such that, for $|z|\leq \epsilon_0$, one has a decomposition
	\begin{equation} \label{eq:P_t-decomp}
	\oP_z = \lambda_z \oN_z + \oQ_z,
	\end{equation}
	where $\lambda_z \in \C$, $\oN_z$ and $\oQ_z$ are bounded operators on $\Cc^{\alpha}(\P^{d-1})$ and 
	\begin{enumerate}
		\item $\lambda_0 = 1$ and $\oN_0 \varphi = \int_{\P^{d-1}} \varphi \, \diff \nu$, which is a constant function,  where $\nu$ is the unique $\mu$-stationary measure;
		\item $\rho:= \displaystyle \lim_{n \to \infty
		 } \|\oP_0^n - \oN_0\|_{\Cc^\alpha}^{1 \slash n} < 1$;
		
		\item $\lambda_z$ is the unique eigenvalue of maximum modulus of $\oP_z$, $\oN_z$ is a rank-one projection and $\oN_z \oQ_z = \oQ_z \oN_z = 0$;
		
		\item the maps $z \mapsto \lambda_z$,  $z \mapsto \oN_z$ and $z \mapsto \oQ_z$ are analytic;
		
		\item  $|\lambda_z| \geq \frac{2 + \rho}{3}$ and for every $k\in\N$, there exists a constant $c > 0$ such that $$\Big \| \frac{\diff^k \oQ_z^n}{\diff z^k} \Big \|_{\Cc^\alpha} \leq c \Big( \frac{1 + 2 \rho}{3} \Big)^n \quad \text{ for every}\quad n \geq 0;$$
		
		\item for $z=i\xi\in i\R$, we have
		 $$   \lambda_{i\xi} = 1 + i  \gamma \xi - \frac{\varrho^2+\gamma^2}{2}\xi^2+O(|\xi|^3) \quad \text {as } \,\, \xi \to 0,$$ where $\gamma$ is the first Lyapunov exponent of $\mu$ and $\varrho > 0$ is a constant.
	\end{enumerate}
\end{proposition}

The constant  $\varrho^2 > 0$ appearing the above expansion of $\lambda_{i\xi}$ coincides with the variance in the Central Limit Theorem for the norm cocycle, see \cite{bougerol-lacroix,benoist-quint:book,DKW:LLT}. As a consequence of the above proposition, we can derive the following estimates which will be crucial in the proof of our main theorems.  For the proof, see \cite[Proposition 8.5]{DKW:LLT} and \cite[Lemma 9]{lepage:theoremes-limites}.
\begin{lemma}\label{lemma:lambda-estimates}
	Let $\ep_0$ be as in Proposition \ref{prop:spectral-decomp}. There exists $0 < \xi_0 < \ep_0$ such that, for all $n \in \N$ large enough, one has $$\big|\lambda_{{i\xi\over \sqrt n}}^n\big|\leq e^{-{\varrho^2\xi^2\over 3}} \quad\text{for}\quad  |\xi|\leq \xi_0\sqrt n,$$
$$\Big|  e^{-i\xi\sqrt n \gamma}\lambda_{{i\xi\over \sqrt n}}^n-e ^{-{\varrho^2\xi^2\over 2}}  \Big|\leq {c\over \sqrt n}|\xi|^3e^{-{\varrho^2\xi^2\over 2}} \quad\text{for}\quad |\xi|\leq \sqrt[6] n,$$  
$$\Big| e^{-i\xi\sqrt n \gamma} \lambda_{{i\xi\over \sqrt n}}^n-e ^{-{\varrho^2\xi^2\over 2}}  \Big|\leq {c\over \sqrt n}e^{-{\varrho^2\xi^2\over 4}} \quad\text{for}\quad \sqrt[6] n<|\xi|\leq \xi_0\sqrt n,$$  
where  $c>0$ is a constant independent of $n$.
\end{lemma}

The following important result describes the spectrum of $\oP_{i\xi}$ for large real values of $\xi$. It is one of the main tools in the proof of the Local Limit Theorem for the norm cocycle and it will also be indispensable in our proof of Theorem \ref{thm:LLT-coeff}.

\begin{proposition}\cite{lepage:theoremes-limites}, \cite[Chapter 15]{benoist-quint:book} \label{prop:spec-Pxi}
Let $\mu$ and $\alpha_0$ be as in Theorem \ref{thm:spectral-gap}. Let $K$ be a compact subset of $\R \setminus \{0\}$. Then, for every $0<\alpha \leq \alpha_0$ there exist constants $C_K>0$ and $0<\rho_K<1$ such that $\norm{\oP^n_{i\xi}}_{\Cc^\alpha}\leq C_K \rho_K^n$ for all $n\geq 1$ and $\xi\in K$.
\end{proposition}

\subsection{Fourier transform and characteristic function}

Recall that the Fourier transform of an integrable function $h$ on $\R$, denoted by $\widehat h$, is defined by
$$\widehat h(\xi):=\int_{-\infty}^{+\infty}h(u)e^{-i u\xi} \diff u$$ 
and the inverse Fourier transform  is $$ \oF^{-1} h(u) :={1\over {2\pi}}\int_{-\infty}^{+\infty} h(\xi) e^{ i u\xi} \diff \xi,$$ so that, when $\widehat h$ is integrable, one has $h = \oF^{-1} \widehat h$. With these definitions, the Fourier transform of $\widehat h(\xi)$ is $2\pi h(-u)$ and the convolution operator satisfies $\widehat{h_1*h_2}=\widehat h_1\cdot \widehat h_2$.

\begin{lemma}[\cite{DKW:LLT}--Lemma 2.2] \label{l:vartheta}
	There exists a smooth strictly positive even function $\vartheta$ on $\R$ with $\int_\R \vartheta(u) \diff u=1$ such that its Fourier transform $\widehat\vartheta$ is  a smooth even function supported by $[-1,1]$.
	Moreover, for $0<\delta \leq 1$ and $\vartheta_\delta(u):=\delta^{-2}\vartheta(u/\delta^2)$, the function $\widehat{\vartheta_\delta}$ is supported by $[-\delta^{-2},\delta^{-2}]$, $|\widehat{\vartheta_\delta}|\leq 1$ and $\norm{\widehat{\vartheta_\delta}}_{\Cc^1}\leq c$   for some constant $c>0$ independent of $\delta$.
\end{lemma}

As a consequence, we have the following approximation lemma. 

\begin{lemma}[\cite{DKW:LLT}--Lemma 2.4] \label{lemma:conv-fourier-approx}
	Let $\psi$ be a continuous real-valued function with support in a compact set $K$ in $\R$.  Assume that $\|\psi\|_\infty \leq 1$. Then, for every  $0< \delta \leq 1$ there exist a smooth functions $\psi^\pm_\delta$  such that $\widehat {\psi^\pm_\delta}$ have support in $[-\delta^{-2},\delta^{-2}]$,  $$\psi^-_\delta \leq\psi\leq \psi^+_\delta,\quad \lim_{\delta \to 0} \psi^\pm_\delta =\psi    \quad \text{and} \quad  \lim_{\delta \to 0} \big \|\psi^\pm_\delta -\psi \big \|_{L^1} = 0.$$ 
	Moreover,  $\norm{\psi_\delta^\pm}_\infty$, $\norm{\psi_\delta^\pm}_{L^1}$ and $\|\widehat{\psi^\pm_\delta}\|_{\Cc^1}$ are bounded by a constant which only depends on $K$.
\end{lemma}

When proving limit theorems for random variables we often resort to the associated characteristic functions. For notational convenience, we will also use their conjugates.

\begin{definition} \label{def:conjugate-cf} \rm 
	For a real random variable $X$ with cumulative distribution function $F$ (c.d.f.\ for short) , we define its \textit{conjugate characteristic function} by $$\phi_F(\xi):=\mathbf E\big(e^{-i\xi X}\big).$$
\end{definition}

Observe that $\diff F$ is a probability measure on $\R$  and $\phi_F$ is its Fourier transform. In particular, when $F$ is differentiable and  $\phi_F$ is integrable, the following inversion formula holds 
\begin{equation}\label{inverse-char}
F'(u)={1\over 2\pi}\int_{-\infty}^\infty  e^{iu\xi} \phi_F(\xi) \,\diff \xi.
\end{equation}

\section{Berry-Esseen bound for coefficients}\label{sec:BE}

This section is devoted to the proof of Theorem \ref{thm:BE-coeff}.  We begin with the following version of Berry-Esseen lemma. See also \cite[XVI.3]{feller:book}.

\begin{lemma}  \label{lemma:BE-feller}
	Let $F$ be a c.d.f.\ of some real random variable and let  $H$  be a differentiable real-valued function  on $\R$ with derivative $h$ such that $H(-\infty)=0,H(\infty)=1,|h(u)|\leq m$ for some constant $m >0$. Let $D>0$ and $0<\delta<1$ be real numbers such that $\big|F(u)-H(u) \big|\leq D \delta^2$ for $|u|\geq \delta^{-2}$.
	Then, there exist constants $C>0$ and $\kappa > 1$ independent of $F,H,\delta$, such that
	$$\sup_{u\in\R}\big|F(u)-H(u) \big|\leq 2\sup_{|u|\leq \kappa \delta^{-2}} \big|(F-H)*\vartheta_\delta(u)\big|+C \delta^2, $$
	where $\vartheta_\delta$ is defined in Lemma \ref{l:vartheta}.
\end{lemma}

\begin{proof}
We begin by noticing that,  from the definition of $\vartheta_\delta$, we have that,  for any $d>0$,
	\begin{equation} \label{eq:BE-lemma-1}
	\int_{|u|\geq d}\vartheta_\delta(u) \,\diff u=\int_{|u|\geq d}{\vartheta(u/\delta^2)\over \delta^2}\,\diff u=\int_{|u|\geq d\delta^{-2}} \vartheta(u)\,\diff u\leq c\delta^2/d
	\end{equation}
	for some constant $c>0$ independent of $d$ and $\delta$. This is due to the fact that $\widehat{\vartheta}$ is smooth and compactly supported, hence $\vartheta$ has fast decay at infinity, say $|\vartheta(u)| \lesssim 1 / |u|^2$. 
	
	Since the function $F(u)-H(u)$ vanishes at $\pm\infty$, the maximum of $\big|F(u)-H(u) \big|$ exists. Let $u_0$ be a point where this maximum is attained.  If $|u_0|\geq \delta^{-2}$,  there is nothing to prove because $\sup_{|u|\geq \delta^{-2}}\big|F(u)-H(u) \big|\leq D \delta^2$ by hypothesis.  So, we can assume $|u_0|\leq \delta^{-2}$ and $M:=\big|F(u_0)-H(u_0) \big|\geq D \delta^2$.  If $M\leq 12 mc\delta^2$, the lemma clearly follows,    so we may  assume $M>12 mc\delta^2$.  We will use the fact that $F(-\infty)=0,F(\infty)=1$ and $F$ is non-decreasing.
	
	After replacing $F(u)$ and $H(u)$ by $1-F(-u)$ and $1-H(-u)$ if necessary, we may assume  that $M=F(u_0)-H(u_0)>0$. Let $d>0$ be a constant such that $M \geq 2  m d$ whose precise value will be determined later.  Since $F$ is non-decreasing and $h(u)\leq m$ by assumption, we have
	$F(u_0+r)-H(u_0+r)\geq M- mr$ for $r\geq 0$.
	Thus,
	\begin{equation} \label{eq:BE-lemma-2}
	F(u)-H(u)\geq M-2md \quad\text{for} \quad  u_0\leq u\leq u_0+2d,
	\end{equation}	
	and from the definition of $M$,
	$$F(u)-H(u)\geq -M  \quad\text{for all }\quad u \in \R.$$

	Therefore, because $|u_0| \leq \delta^{-2}$ , we obtain using \eqref{eq:BE-lemma-1} and \eqref{eq:BE-lemma-2} that
	\begin{align*}
	\sup_{|u|\leq \delta^{-2}+d} \big|(F-H)*\vartheta_\delta(u)\big| &\geq  (F-H)*\vartheta_\delta(u_0+d)\\
	&=\Big(\int_{|u|< d}+ \int_{|u|\geq d}\Big) (F-H)(u_0+d-u) \cdot \vartheta_\delta (u)\,\diff u\\
	&\geq (M-2md)(1-c \delta^2 /d)  - Mc \delta^2/d \\ &=(1-2c \delta^2/d)M-2md+2m c \delta^2.
	\end{align*}
	 By setting $d:=4c\delta^2$ and recalling that  $M>12 mc\delta^2$, we get that $M  \geq 2md$ and the last quantity above equals $M/2-6mc\delta^2$. Since $\delta^{-2}+d \leq (1+4c) \delta^{-2}$, the lemma follows by setting $\kappa := 1+4c$ and $C:= 12 mc$. 
\end{proof}

\begin{corollary} \label{cor:BE-feller}
Keep the notations and assumptions of Lemma \ref{lemma:BE-feller}. Assume moreover that $h\in L^1$, $\widehat h\in\Cc^1$ and that $\phi_F$ is differentiable at zero (see Definition \ref{def:conjugate-cf}).  Then,
	$$\sup_{u\in\R}\big|F(u)-H(u) \big|\leq {1\over \pi} \sup_{|u|\leq \kappa \delta^{-2}}    \Big|\int_{-\delta^{-2}}^{\delta^{-2}} {\Theta_u(\xi) \over \xi}     \,\diff \xi\Big|  +C\delta^2,$$
	where  $\Theta_u(\xi):=e^{iu\xi}\big(\phi_F(\xi)- \widehat {h}(\xi) \big)\widehat{\vartheta_\delta}(\xi)$.
\end{corollary}
\begin{proof}
 	Notice that, by the convolution formula, the function $\phi_F \cdot\widehat{\vartheta_\delta}$ is the conjugate characteristic function associated with the c.d.f.\ $F*\vartheta_\delta$. Since $\supp(\widehat{\vartheta_\delta})\subset [-\delta^{-2},\delta^{-2}]$, and $\phi_F$ is bounded by definition, it follows that $\phi_F \cdot\widehat{\vartheta_\delta}$ is integrable.  Identity \eqref{inverse-char} gives that
	$$ (F*\vartheta_\delta)'(u) ={1\over 2\pi}\int_{-\infty}^\infty  e^{iu\xi} \phi_F(\xi)\cdot\widehat{\vartheta_\delta}(\xi) \,\diff \xi.      $$
	As the inverse Fourier transform of $\widehat{h} \cdot \widehat{\vartheta_\delta}$ is $h*\vartheta_\delta$, we get
	$$\big((F-H)*\vartheta_\delta\big)'(u)={1\over 2\pi}\int_{-\delta^{-2}}^{\delta^{-2}} e^{iu\xi} \big( \phi_F(\xi)-\widehat h(\xi)\big) \widehat{\vartheta_\delta}(\xi) \,\diff\xi.$$
	
	Observe that $\phi_F(0) = \E(\mathbf 1) = 1$ and $\widehat h(0) = H(\infty) - H(-\infty) = 1$, so $\phi_F(0)-\widehat h(0)=0$. Moreover, $\phi_F'(0)-\widehat h\,'(0)$ is finite by the assumptions on $F$ and $h$.  Integrating the above identity with respect to $u$ yields $$(F-H)*\vartheta_\delta (u)={1\over 2\pi}\int_{-\delta^{-2}}^{\delta^{-2}} {e^{iu\xi}\over i\xi} \big( \phi_F(\xi)-\widehat h(\xi)\big) \widehat{\vartheta_\delta}(\xi) \,\diff\xi .$$
	Here, the constant term is zero because, when $u\to\pm\infty$, the left hand side tends to zero and, by the above observations, the integrand is a bounded function, so the integral in the right hand side also tends to zero as $u\to\pm\infty$ by Riemann–Lebesgue lemma. The desired result follows from Lemma \ref{lemma:BE-feller}.
\end{proof}

\medskip
 
   Fix $x:=[v]\in \P^{d-1}, y:=[f]\in (\P^{d-1})^*$ and consider the pairing $$\Delta(x,y):= \frac{ |\lp f, v \rp |}{\norm{f} \norm{v}},$$ where $\lp \cdot ,  \cdot \rp $ denotes the natural pairing between $\R^d$ and $(\R^d)^*$.   One can easily check that $\Delta(x,y) = d (x,H_y)$, where $H_y := \P(\ker f)$ and $d$ is the distance defined in Section \ref{sec:prelim}. Using these definitions, it is not hard to see that
\begin{equation}\label{eq:coeff-split}
\log{ |\lp f, S_n v \rp | \over \norm{f} \norm{v}} = \sigma(S_n,x) + \log d(S_n x, H_y).
\end{equation}

The strategy to prove Theorem \ref{thm:BE-coeff} is to use the above formula and work with the random variable $\sigma(S_n,x) +\log d(S_n x, H_y)$ instead of $\log{ |\lp f, S_n v \rp | \over \norm{f} \norm{v}} $. Then, the behaviour of $\sigma(S_n,x)$ can be studied via the perturbed Markov operators and the term  $\log d(S_n x, H_y)$ can be handled using the large deviation estimates from Section \ref{sec:prelim} combined with a good partition of unity that we now introduce.

\medskip

For integers $k \geq 0$ introduce
\begin{align*}
\Tc_k := \big\{ w \in \P^{d-1} :\, e^{-k-1} < d(w,H_y) < e^{-k+1} \big\} = \B(H_y,e^{-k+1}) \setminus \overline{\B(H_y,e^{-k-1})} .
\end{align*}
Note that,  since $\P^{d-1}$ has diameter one,  these open sets cover $\P^{d-1}$.

\begin{lemma} \label{lemma:partition-of-unity}
There exist non-negative smooth functions $\chi_k$ on $\P^{d-1}$, $k \geq 0$,  such that
\begin{enumerate}
\item $\chi_k$ is supported by $\Tc_k$;
\item If $w \in \P^{d-1} \setminus H_y$, then  $\chi_k(w) \neq 0$ for at most two values of $k$;
\item $\sum_{k\geq 0}  \chi_k=1$ on $\P^{d-1} \setminus H_y$; 
\item $\norm{\chi_k}_{\Cc^1}\leq 12e^{k}$.
\end{enumerate}
\end{lemma}

\begin{proof}
It is easy to find a smooth function $0 \leq \widetilde \chi \leq 1$ supported by $(-1,1)$ such that $\widetilde \chi(t) = 1$ for $|t|$ small, $\widetilde \chi(t) + \widetilde \chi(t-1)= 1$ for $0 \leq t\leq 1$ and $\norm{\widetilde \chi}_{\Cc^1}\leq 4$.  Define  $\widetilde \chi_k (t) := \widetilde \chi(t+k)$. We see that $\widetilde \chi_k$ is supported by $(-k-1,-k+1)$,  $\sum_{k\geq 0} \widetilde \chi_k=1$ on $\R_{\leq 0}$ and $\norm{\widetilde \chi_k}_{\Cc^1}\leq 4$. Set $\chi_k(w):= \widetilde \chi_k \big( \log d(w,H_y) \big)$.  One can easily check that the function $\Psi(w) := \log d(w,H_y)$ satisfies $\norm{\Psi|_{\Tc_k }}_{\Cc^1} \leq e^{k+1}$.  It follows that $\chi_k$ satisfies (1)--(4).
\end{proof}

We now begin the proof of Theorem \ref{thm:BE-coeff}.  It  suffices to prove  Theorem \ref{thm:BE-coeff}  for intervals of the type $J=(-\infty,b]$ with $b \in \R$, as the case of an arbitrary interval can be obtained as a consequence. For example, the case $(b,+\infty)$ follows directly by considering its complement. The case of $[b, +\infty)$ can be deduced by approximating it by $(b \pm \varepsilon, + \infty)$ and the case $(-\infty,b)$  follows  by taking the complement. The case of bounded intervals can be obtained by considering differences of the previous cases. 

	Let $A>0$ be a large constant. By Proposition \ref{prop:BQLDT} applied for $\ep = 1$, there exists a constant $c>0$ such that with $\ell,m$  large enough and $\ell\geq m$, one has
		$$   \mu^{*\ell} \big\{g\in G:\, d(gx, H_y) \leq e^{- m}  \big\}  \leq e^{-cm}.         $$
		Setting $\ell:=n$ and $m:=\lfloor A \log n \rfloor$ with $n$ big enough yields
		$$    \mu^{*n} \big\{g\in G:\,d(gx,H_y)\leq  n^{-A}       \big\}\leq e^{-c\lfloor A \log n \rfloor} \leq n^{-cA}e^c \leq  e^c/\sqrt n,$$
		since $A$ is large.
		It follows that $\log d(S_n x,H_y)\leq -A\log n $  with probability less than $e^c/\sqrt n$. Hence, in order to prove Theorem \ref{thm:BE-coeff}, it is enough to show that
		\begin{equation}\label{goal-1-varphi}
		\Big| \oL_n(b)  -  \frac{1}{\sqrt{2 \pi}  \,\varrho}  \int_{-\infty}^b e^{-\frac{s^2}{2 \varrho^2}} \, \diff s \Big| \lesssim \frac{1}{\sqrt n},
		\end{equation}
uniformly in $b$, where
$$\oL_n(b):=  \mathbf E \Big(  \mathbf 1_{{\sigma(S_n,x) +\log d(S_n x, H_y) - n \gamma\over \sqrt n}\leq b}   \mathbf 1_{ \log d(S_n x,H_y)> - A \log n  } \Big),$$   
where we use $\mathbf 1_\bigstar$ to denote the indicator function of a set defined by the property $\bigstar$.

Let $\chi_k$ be as in Lemma \ref{lemma:partition-of-unity}.  It is clear that
$$\sum_{0\leq k\leq A\log n - 1}\chi_k(w) \leq \mathbf 1_{  \log d(w,H_y)> - A \log n} \leq  \sum_{0\leq k\leq A\log n+1}\chi_k(w)$$ as functions in $\P^{d-1}$. Using that $\chi_k$ is supported by $\Tc_k$,  it follows that 
\begin{align} \label{eq:Ln-two-sided-bound}
\sum_{0\leq k\leq A\log n-1} \hspace{-7pt} \E\Big(  \mathbf 1_{{\sigma(S_n,x) - n \gamma - k+1\over \sqrt n}\leq b} \, \chi_k(S_n x) \Big) \leq \oL_n(b)\leq \hspace{-5pt} \sum_{0\leq k\leq A\log n+1} \hspace{-7pt}  \mathbf E\Big(  \mathbf 1_{{\sigma(S_n,x) - n \gamma - k-1\over \sqrt n}\leq b} \, \chi_k(S_n x) \Big).
\end{align}

For $w \in \P^{d-1}$, let $$\Phi_{n}^{\star}  (w):= 1 - \sum_{0\leq k\leq A\log n} \chi_k(w)$$ and define, for $b \in \R$,
\begin{align*}
F_n(b):&=    \sum_{0\leq k\leq A\log n} \E\Big(  \mathbf 1_{{\sigma(S_n,x) - n \gamma - k\over \sqrt n}\leq b}  \chi_k (S_n x) \Big)  +  \E \Big( \mathbf 1_{{\sigma(S_n,x) - n \gamma \over \sqrt n}\leq b} \, \Phi_{n}^{\star} (S_n x) \Big). 
\end{align*}

Notice that $F_n$ is non-decreasing, right-continuous, $F_n(-\infty)=0$ and $F_n(\infty)=1$. Therefore, it is the c.d.f.\ of some probability distribution.  We'll see that the term involving $\Phi_{n}^{\star}$ has a negligible impact in our estimates. However, its presence is important and will be useful in our computations.

\begin{lemma} \label{lemma:Ln-Fn}
Let $\oL_n$ and $F_n$ be as above. Then, there exists a constant $C>0$ independent of $n$ such that for all $n \geq 1$ and $b \in \R,$ $$ F_n(b - 1 / \sqrt n)  + C / \sqrt n \leq \oL_n(b) \leq F_n(b + 1 / \sqrt n)  + C / \sqrt n.$$
\end{lemma}

\begin{proof}
Notice first that $\Phi_{n}^{\star}$ is non-negative, bounded by one and supported by a tubular neighborhood $\mathbf T_n$ of $H_y$ of radius $O(n^{-A})$. As discussed above, the probability that $S_n x$ belongs to $\mathbf T_n$ is $\lesssim 1 / \sqrt n$. This yields the following bounds for the second term in the definition of $F_n$:  $$0 \leq \E \Big( \mathbf 1_{{\sigma(S_n,x) - n \gamma \over \sqrt n}\leq b} \, \Phi_{n}^{\star} (S_n x) \Big) \leq \E \Big(\Phi_{n}^{\star}(S_n x) \Big)  \lesssim 1 / \sqrt n.$$

Therefore, in order to prove the lemma, we can replace $F_n$ by the function
\begin{equation} \label{eq:def-Fn-tilde}
\widetilde F_n(b) := \sum_{0\leq k\leq A\log n} \E\Big(  \mathbf 1_{{\sigma(S_n,x) - n \gamma - k \over \sqrt n}\leq b}  \chi_k (S_n x) \Big).
\end{equation}

Using the second inequality in \eqref{eq:Ln-two-sided-bound}, we have
\begin{align*}
\oL_n(b) - \widetilde F_n(b + 1 / \sqrt n)  \leq   \E\Big(  \mathbf 1_{{\sigma(S_n,x) - n \gamma - k^+-1\over \sqrt n}\leq b}  \chi_{k^+} (S_n x) \Big) \leq   \E\Big( \chi_{k^+} (S_n x) \Big),
\end{align*}
where $k^+:= \lfloor A \log n \rfloor + 1$. Since  $\chi_k \leq \mathbf 1_{\B(H_y,e^{-k+1})}$ and  $\log d(S_n x,H_y)\leq -A\log n +1 $  with probability $\lesssim 1/\sqrt n$, the above quantity is $\lesssim 1/\sqrt n$. This gives the second inequality in the lemma.

Using now the first inequality in \eqref{eq:Ln-two-sided-bound} and letting $k^-:= \lfloor A \log n \rfloor$, we obtain
\begin{align*}
 \widetilde F_n(b -  1 /\sqrt n)  - \oL_n(b) \leq   \E\Big(  \mathbf 1_{{\sigma(S_n,x) - n \gamma - k^- +1\over \sqrt n}\leq b}  \chi_{k^-} (S_n x) \Big) \leq   \E\Big(\chi_{k^-} (S_n x) \Big),
\end{align*}
which is $\lesssim 1/\sqrt n$ by the same arguments as before. The lemma follows.
\end{proof}

Introduce $$\Phi_{n,\xi} (w):= \sum_{0\leq k\leq A\log n}  e^{i \xi{k \over \sqrt n}}\chi_k(w).$$ 

\begin{lemma} \label{lemma:char-function-Fn}
The conjugate characteristic function of $F_n$ (cf.\ Definition \ref{def:conjugate-cf}) is given by
$$\phi_{F_n}(\xi)=  e^{i\xi\sqrt n \gamma} \oP_{-{i\xi\over \sqrt n}}^n \Phi_{n,\xi} (x) + e^{i\xi\sqrt n \gamma} \oP_{-{i\xi\over \sqrt n}}^n \Phi_{n}^{\star}   (x) = e^{i\xi\sqrt n \gamma} \oP_{-{i\xi\over \sqrt n}}^n \big( \Phi_{n,\xi} +  \Phi_{n}^{\star} \big)  (x).$$
In particular,  $\phi_{F_n}$ is differentiable near zero.
\end{lemma}

\begin{proof}
 Recall that $x$ is fixed. Let $c_{k,n}:= \int_G \chi_k(gx) \, \diff \mu^{\ast n} (g)$ and $\mu_{k,n}: =  c^{-1}_{k,n} \, \chi_k(gx) \,  \mu^{\ast n}$, which is a probability measure on $G$ that is absolutely continuous with respect to $\mu^{\ast n}$. Let $Z_{n,k}$ be the measurable function ${\sigma(g,x) - n \gamma - k \over \sqrt n}$  on the probability space $(G,\mu_{k,n})$. The corresponding c.d.f.\ is  $$F_{Z_{n,k}}(b) = c^{-1}_{k,n} \int_G \mathbf 1_{{\sigma(g,x) - n \gamma - k\over \sqrt n}\leq b}   \chi_k(gx) \, \diff \mu^{\ast n} (g)$$
 and the associated conjugate characteristic function is $$\phi_{F_{Z_{n,k}}}(\xi) =  c^{-1}_{k,n} \int_G e^{-i \xi{\sigma(g,x) - n \gamma - k \over \sqrt n}} \chi_k(gx) \, \diff \mu^{\ast n} (g) =  c^{-1}_{k,n} e^{i\xi\sqrt n \gamma} \oP_{-{i\xi\over \sqrt n}}^n  \big(e^{i \xi{k\over \sqrt n}}\chi_k \big) (x),$$
 where we have used \eqref{eq:markov-op-iterate}.

Analogously, set $d_n:= \int_G \Phi_{n}^{\star}(gx) \, \diff \mu^{\ast n}(g)$, consider the probability measure $\eta_{n}: =  d^{-1}_{n} \, \Phi_{n}^{\star} (gx) \,  \mu^{\ast n}$  and let  $W_{n}$ be the measurable function ${\sigma(g,x) - n \gamma \over \sqrt n}$  on the probability space $(G,\eta_{n})$. Then, the corresponding c.d.f.\ is  $$F_{W_n}(b) = d^{-1}_{k} \int_G \mathbf  \mathbf 1_{{\sigma(g,x) - n \gamma \over \sqrt n}\leq b}    \Phi_{n}^{\star}(gx)  \, \diff \mu^{\ast n} (g)$$
and the associated conjugate characteristic function is 
\begin{align*}
\phi_{F_{W_n}}(\xi) =  d^{-1}_{n} \int_G e^{-i \xi{\sigma(g,x) - n \gamma \over \sqrt n}}  \Phi_{n}^{\star}(gx) \, \diff \mu^{\ast n} (g)  =   d^{-1}_{n} e^{i\xi\sqrt n \gamma} \oP_{-{i\xi\over \sqrt n}}^n \Phi_{n}^{\star}   (x).
\end{align*}

Notice that, by definition $F_n = \sum_{0\leq k\leq A\log n} c_{k,n}  F_{Z_{n,k}} + d_n  F_{W_n}$ so, by linearity, $\phi_{F_n} = \sum_{0\leq k\leq A\log n} c_{k,n} \phi_{F_{Z_{n,k}}} + d_n \phi_{F_{W_n}}$. Using the definition of $\Phi_{n,\xi}$, the lemma follows.
\end{proof}

From now on, we fix the value of the constant $A >0$ used above. Fix a value of $\alpha>0$ such that $$\alpha A\leq 1/6 \quad \text{ and } \quad \alpha \leq \alpha_0,$$ where $0<\alpha_0<1$ is the exponent appearing in Theorem \ref{thm:spectral-gap}. Then, from the results of Subsection \ref{subsec:markov-op}, the family $\xi\mapsto\oP_{i\xi}$ acting on $\Cc^\alpha(\P^{d-1})$ with $\xi \in \R$ is everywhere defined, analytic near $0$ and $\oP_0$ has a spectral gap. The bound $\alpha A\leq 1/6$ is chosen so that we can control the impact of the H\"older norms of $\Phi_{n,\xi}$ and  $\Phi_{n}^{\star}$ in our estimates, as shown in the following lemma. 

\begin{lemma}  \label{lemma:norm-Phi}
Let $\Phi_{n,\xi}, \Phi_{n}^{\star}$ be the functions on $\P^{d-1}$ defined above.  Then,  the following identity holds
\begin{equation} \label{eq:psi_xi+psi_T}
\Phi_{n,\xi} + \Phi_{n}^{\star} = \mathbf 1 + \sum_{0\leq k\leq A\log n} \big(  e^{i \xi{k \over \sqrt n}} - 1 \big) \chi_k.
\end{equation}
Moreover,  $\norm{\Phi_{n,\xi} + \Phi_{n}^{\star}}_\infty \leq 1$ and there  is a constant $C>0$ independent of $\xi$ and $n$ such that   
\begin{equation}  \label{eq:norm-Phi}
\norm{\Phi_{n,\xi} }_{\Cc^\alpha}\leq C \,  n^{\alpha A} \quad\text{and}\quad  \norm{\Phi_{n}^{\star}  }_{\Cc^\alpha} \leq C \,  n^{\alpha A},
\end{equation}
where $\alpha>0$ is the exponent fixed above.  In addition,  $\Phi_{n}^{\star}  $ is supported by $\big\{w:\,\log d(w,H_y)\leq -A\log n + 1\big\}$.
\end{lemma}

\begin{proof}

Identity \eqref{eq:psi_xi+psi_T} follows directly from the definition of $\Phi_{n,\xi}$ and $ \Phi_{n}^{\star}$.   Also from the definition,   we have  $|\Phi_{n,\xi} + \Phi_{n}^{\star} | \leq \Phi_{n,0} + \Phi_{n}^{\star}$ and the last function is identically equal to $1$ by  \eqref{eq:psi_xi+psi_T}.  It follows that $\norm{\Phi_{n,\xi} + \Phi_{n}^{\star}}_\infty \leq 1$.

 For the first inequality in \eqref{eq:norm-Phi}, notice that $\big \| e^{i \xi{k \over \sqrt n}}\chi_k \big\|_{\Cc^\alpha} = \norm{\chi_k}_{\Cc^\alpha}$. From Lemma \ref{lemma:partition-of-unity}-(4), the fact that $\norm{\chi_k}_{\Cc^0} \leq 1$ and the interpolation inequality $\norm{\, \cdot \,}_{\Cc^\alpha} \leq c_\alpha \norm{\, \cdot \,}_{\Cc^0}^{1-\alpha} \, \norm{\, \cdot \,}_{\Cc^1}^\alpha$ (see \cite[p. 202]{triebel}), it follows that $\norm{\chi_k}_{\Cc^\alpha} \leq 12 c_\alpha e^{\alpha k}$.  The last inequality can also be checked by a direct computation. Then, the first inequality in \eqref{eq:norm-Phi} follows from the definition of $\Phi_{n,\xi}$ and the fact that at most two $\chi_k$'s are non-zero simultaneously. The second inequality in \eqref{eq:norm-Phi} follows from the first one and the identity  $\Phi_{n,0} + \Phi_{n}^{\star} = \mathbf 1$,  after increasing the value of $C$ if necessary. 

In order to prove the last assertion, observe that, over $\P^{d-1} \setminus H_y$, one has $\Phi_{n}^{\star} = \sum_{k > A\log n} \chi_k$ by Lemma \ref{lemma:partition-of-unity}-(3). Since $\chi_k$ is supported by $\Tc_k$, the conclusion follows. This finishes the proof of the lemma.
\end{proof}

	Let $$H(b): =\frac{1}{\sqrt{2 \pi}  \,\varrho} \int_{-\infty}^b e^{-\frac{s^2}{2 \varrho^2}} \, \diff s$$ be the c.d.f.\ of the normal distribution $\cali N(0;\varrho^2)$. In the notation of Lemma \ref{lemma:BE-feller}, we have $h(u): =\frac{1}{\sqrt{2 \pi}  \,\varrho}   e^{-\frac{u^2}{2 \varrho^2}}$ and $\widehat h(\xi) =  e^{-{\varrho^2 \xi^2\over 2}}$. Let $\xi_0$ be the constant in Lemma \ref{lemma:lambda-estimates}. 
  \medskip

\begin{lemma} \label{lemma:F_n-estimate}
Let $F_n$ and $H$ be as above. Then, $\big|F_n(b)-H(b)\big|\lesssim 1/\sqrt n$ for $|b|\geq \xi_0 \sqrt n$.
\end{lemma}

\begin{proof}
We only consider the case of $b\leq -\xi_0\sqrt n$. The case $b \geq \xi_0\sqrt n$ can be treated similarly using $1-F_n$ and $1-H$ instead of $F_n$ and $H$. We can also assume that $n$ is large enough. Clearly, $H(b)\lesssim 1/\sqrt n$ for $b\leq -\xi_0\sqrt n$,  so it is enough to bound $F_n(b)$.  

 For $0\leq k\leq A\log n$, we have
\begin{align*}
\mathbf P\Big(  {\sigma(S_n,x) - n \gamma - k \over \sqrt n}\leq -\xi_0\sqrt n  \Big)&=\mathbf P\Big(\sigma(S_n,x) - n \gamma  \leq -\xi_0 n +k  \Big)\\
&\leq \mathbf P\Big(\sigma(S_n,x) - n \gamma  \leq  -  \xi_0 n / 2\Big),
\end{align*}
since $n$ is large. By Proposition \ref{prop:BQLDT} applied with $\ep= \xi_0 / 2$, there exists a constant $c>0$, independent of $n$, such that the last quantity is bounded by $e^{-cn}$.

Using the definition of $F_n$ and the fact that $\E \big(\Phi_{n}^{\star}(S_n x) \big)  \lesssim 1 / \sqrt n$ (see the proof of Lemma \ref{lemma:Ln-Fn}),  it follows that $$ F_n (-\xi_0\sqrt n)\lesssim \sum_{0\leq k\leq A\log n}e^{-cn} + 1/\sqrt n  \lesssim (A \log n) e^{-cn} + 1/\sqrt n \lesssim 1/\sqrt n.$$ As $F_n(b)$ is non-decreasing in $b$, one gets that $F_n(b) \lesssim 1 / \sqrt n$ for all $b\leq -\xi_0\sqrt n$. The lemma follows.
\end{proof}

Lemmas \ref{lemma:char-function-Fn} and \ref{lemma:F_n-estimate}  imply that $F_n$ satisfies the conditions of Corollary \ref{cor:BE-feller} with $\delta_n:=(\xi_0 \sqrt n)^{-1/2}$.  Let $\kappa > 1$ be the constant appearing in that corollary. For simplicity, by taking a smaller $\xi_0$ is necessary, one can assume that $2\kappa \xi_0 \leq 1$.  Then,  Corollary \ref{cor:BE-feller} gives that
\begin{equation} \label{eq:BE-main-estimate}
\sup_{b \in\R}\big|F_n(b)-H(b)\big|\leq  {1\over \pi} \sup_{|b|\leq \sqrt n}    \Big|  \int_{-\xi_0\sqrt n}^{\xi_0\sqrt n} {\Theta_b(\xi) \over \xi}    \,\diff \xi  \Big|+{C\over \sqrt n},
\end{equation}
where $C>0$ is a constant independent of $n$ and 
$$\Theta_b(\xi):=e^{ib\xi}\big( \phi_{F_n}(\xi)- \widehat {h}(\xi) \big)\widehat{\vartheta_{\delta_n}}(\xi).$$

We now estimate the integral in the right hand side of \eqref{eq:BE-main-estimate}.  Define
$$\widetilde h_n(\xi):=  \big(\oN_0 \Phi_{n,\xi} \big)  e^{-{\varrho^2 \xi^2\over 2}} +    \big(\oN_0 \Phi_{n}^{\star} \big)    e^{-{\varrho^2 \xi^2\over 2}} =  e^{-{\varrho^2 \xi^2\over 2}}  \oN_0 (  \Phi_{n,\xi} +  \Phi_{n}^{\star} ).$$ In light of Lemma \ref{lemma:char-function-Fn}, we will use it to approximate $\phi_{F_n}$ (see Proposition \ref{prop:spectral-decomp} and Lemma \ref{lemma:lambda-estimates}). Notice that  $\oN_0 (  \Phi_{n,\xi} +  \Phi_{n}^{\star} )$ is a constant independent of $x$.  Define also
 $$\Theta_b^{(1)}(\xi):= e^{ib\xi}\big( \phi_{F_n}(\xi)- \widetilde {h}_n(\xi) \big)\widehat{\vartheta_{\delta_n}}(\xi) \quad\text{and} 
 \quad \Theta_b^{(2)}(\xi):= e^{ib\xi}\big(  \widetilde {h}_n(\xi)-\widehat h(\xi) \big)\widehat{\vartheta_{\delta_n}}(\xi),$$
so that $\Theta_b =  \Theta_b^{(1)} + \Theta_b^{(2)}$. 
\medskip

\begin{lemma} \label{lemma:theta-1-bound} We have
$$\sup_{|b|\leq \sqrt n}    \Big|  \int_{-\xi_0\sqrt n}^{\xi_0\sqrt n} {\Theta_b^{(1)}(\xi) \over \xi}    \,\diff \xi  \Big| \lesssim {1\over \sqrt n}.$$
\end{lemma}

\begin{proof}
Using Lemma \ref{lemma:char-function-Fn} and the decomposition in Proposition \ref{prop:spectral-decomp}, we write
$$\Theta_b^{(1)}=\Lambda_1+\Lambda_2+\Lambda_3,$$
where 
$$ \Lambda_1(\xi):= e^{ib\xi}\Big(  e^{i\xi\sqrt n \gamma} \lambda_{-{i\xi\over \sqrt n}}^n \oN_{-{i\xi\over \sqrt n}} (\Phi_{n,\xi} +\Phi_{n}^{\star}  ) (x) -      e^{-{\varrho^2 \xi^2\over 2}}   \oN_0 (\Phi_{n,\xi} +\Phi_{n}^{\star}  )  (x)   \Big) \widehat{\vartheta_{\delta_n}}(\xi), $$ 
$$\Lambda_2(\xi):=e^{ib\xi}\Big(  e^{i\xi\sqrt n \gamma} \oQ_{-{i\xi\over \sqrt n}}^n  (\Phi_{n,\xi} +\Phi_{n}^{\star}  )(x) -e^{i\xi\sqrt n \gamma}\oQ_0^n  (\Phi_{n,\xi} +\Phi_{n}^{\star}  )(x)  \Big)\widehat{\vartheta_{\delta_n}}(\xi)$$
and
$$ \Lambda_3(\xi):=e^{ib\xi}e^{i\xi\sqrt n \gamma}\oQ_0^n  (\Phi_{n,\xi} +\Phi_{n}^{\star}  )(x)  \,\widehat{\vartheta_{\delta_n}}(\xi).$$

We will estimate the integral  of $\Lambda_j(\xi) / \xi$,  $j=1,2,3$, separately.  Notice that, from \eqref{eq:psi_xi+psi_T}, we have that $\Phi_{n,0} + \Phi_{n}^{\star} = \mathbf 1 $.   Together with the fact that $\lambda_0=1$, $\oN_0 \mathbf 1 = \mathbf 1$ and $\oQ_0 \mathbf 1 = 0$, we get that  $\Lambda_j(0) = 0$ for  $j=1,2,3$.  In particular, $\Lambda_j(\xi) / \xi$ is a smooth function of $\xi$ for $j=1,2,3$. We see here the role of the auxiliary function $\Phi_{n}^{\star}$.

\vskip 5pt

In order to estimate $\Lambda_2$ observe that, for $z$ small, the norm of the operator $\oQ^n_z - \oQ^n_0$ is bounded by a constant times $|z| n \beta^n$ for some $0<\beta<1$. This can be seen by writing the last difference as $\sum_{\ell=0}^{n-1}  \oQ_z^{n-\ell-1}(\oQ_z - \oQ_0)\oQ_0^\ell$, applying Proposition \ref{prop:spectral-decomp}-(5) and using the fact that $\norm{\oQ_z - \oQ_0}_{\Cc^\alpha} \lesssim |z|$ . Therefore, we have 
 $$  \Big|  \oQ_{-{i\xi\over \sqrt n}}^n  (\Phi_{n,\xi} +\Phi_{n}^{\star}  )(x) - \oQ_0^n (\Phi_{n,\xi} +\Phi_{n}^{\star}  )(x) \Big| \lesssim {|\xi|\over \sqrt n} n \beta^n \norm{ \Phi_{n,\xi} +\Phi_{n}^{\star}  }_{\Cc^\alpha}.$$
 Using Lemma \ref{lemma:norm-Phi}, this gives
 \begin{align*}
\Big|\int_{-\xi_0\sqrt n}^{\xi_0\sqrt n} {\Lambda_2(\xi) \over \xi}    \,\diff \xi\Big|
&\leq\int_{-\xi_0\sqrt n}^{\xi_0\sqrt n} {1\over |\xi|} \cdot \Big| \oQ_{-{i\xi\over \sqrt n}}^n  (\Phi_{n,\xi} +\Phi_{n}^{\star}  )(x) -  \oQ_0^n (\Phi_{n,\xi} +\Phi_{n}^{\star}  )(x) \Big|\,\diff\xi \\
&\lesssim     \int_{-\xi_0\sqrt n}^{\xi_0\sqrt n} {1\over |\xi|} \cdot  |\xi| \sqrt n \beta^n \norm{ \Phi_{n,\xi} +\Phi_{n}^{\star}  }_{\Cc^\alpha} \,\diff\xi \lesssim \beta^n n^{\alpha A + 1} \lesssim {1\over \sqrt n}.     
\end{align*}

We now estimate $\Lambda_3$ using its derivative $\Lambda'_3$ . Recall that $\norm{\widehat{\vartheta_{\delta_n}}}_{\Cc^1}\lesssim 1,|b|\leq \sqrt n$ and  $\big|\oQ_0^n  (\Phi_{n,\xi} +\Phi_{n}^{\star}  )(x) \big|\lesssim \beta^n\norm{ \Phi_{n,\xi} +\Phi_{n}^{\star}  }_{\Cc^\alpha}$ where $0<\beta<1$ is as before. A direct computation using the definition of $\Phi_{n,\xi} $ gives
\begin{align*}
\sup_{|\xi|\leq \sqrt n} | \Lambda_3'(\xi)| &\leq  \Big|b \oQ_0^n  (\Phi_{n,\xi} +\Phi_{n}^{\star}  )(x)\Big|+    \Big|\sqrt n \gamma \oQ_0^n  (\Phi_{n,\xi} +\Phi_{n}^{\star}  )(x)\Big|  \\ & \quad \quad + \sum_{1 \leq k \leq A \log n} \Big|{k\over \sqrt n} e^{i \xi \frac{k}{\sqrt n}}\oQ_0^n \chi_k (x)    \Big|  + \Big|\oQ_0^n  (\Phi_{n,\xi} +\Phi_{n}^{\star}  )(x)\Big| \\ &\lesssim \sqrt n  \beta^n n^{\alpha A} + \frac{(\log n)^2}{\sqrt n} \beta^n n^{\alpha A} +   \beta^n n^{\alpha A}  \lesssim (1+\sqrt n)  \beta^n n^{\alpha A},
\end{align*}
where we have used that $\norm{\chi_k}_{\Cc^\alpha} \lesssim e^{\alpha k}$ and $\norm{ \Phi_{n,\xi} +\Phi_{n}^{\star}  }_{\Cc^\alpha} \lesssim n^{\alpha A}$, see Lemma \ref{lemma:norm-Phi}.

Applying the mean value theorem  over the interval between $0$ and $\xi$ yields
 \begin{align*}
\Big|\int_{-\xi_0\sqrt n}^{\xi_0\sqrt n} {\Lambda_3(\xi) \over \xi}    \,\diff \xi\Big|\leq 2\xi_0\sqrt n\sup_{|\xi|\leq \xi_0\sqrt n} | \Lambda_3'(\xi)|\lesssim \sqrt n (1+\sqrt n) n^{\alpha A} \beta^n \lesssim {1 \over \sqrt n}.
\end{align*}
\vskip 5pt

It remains to estimate the term involving $\Lambda_1$. We have
$$ \Big|\int_{-\xi_0\sqrt n}^{\xi_0\sqrt n} {\Lambda_1(\xi) \over \xi}   \diff \xi\Big| \leq \int_{-\xi_0\sqrt n}^{\xi_0\sqrt n} {1\over |\xi|}\cdot \Big|  e^{i\xi\sqrt n \gamma} \lambda_{-{i\xi\over \sqrt n}}^n \oN_{-{i\xi\over \sqrt n}} (\Phi_{n,\xi} +\Phi_{n}^{\star}  ) (x) -      e^{-{\varrho^2 \xi^2\over 2}}   \oN_0 (\Phi_{n,\xi} +\Phi_{n}^{\star}  )(x)    \Big|   \diff \xi.$$
We split the last integral into two integrals using
$$ \Gamma_1(\xi):=   e^{i\xi\sqrt n \gamma} \lambda_{-{i\xi\over \sqrt n}}^n \oN_{-{i\xi\over \sqrt n}}  (\Phi_{n,\xi} +\Phi_{n}^{\star}  ) (x) -      e^{i\xi\sqrt n \gamma} \lambda_{-{i\xi\over \sqrt n}}^n  \oN_0 (\Phi_{n,\xi} +\Phi_{n}^{\star}  )(x)   $$
and
$$  \Gamma_2(\xi):=   e^{i\xi\sqrt n \gamma} \lambda_{-{i\xi\over \sqrt n}}^n  \oN_0 (\Phi_{n,\xi} +\Phi_{n}^{\star}  ) (x) -      e^{-{\varrho^2 \xi^2\over 2}}   \oN_0 (\Phi_{n,\xi} +\Phi_{n}^{\star}  ) (x) .   $$
\vskip 5pt

\noindent\textbf{Case 1.}  $\sqrt[6] n<|\xi|\leq \xi_0\sqrt n$.    In this case, by Lemma \ref{lemma:lambda-estimates} we have
\begin{equation} \label{eq:case1-lambda}
\big|\lambda_{-{i\xi\over \sqrt n}}^n\big|\leq e^{-{\varrho^2\xi^2\over 3}} \quad \text{and}\quad \Big| e^{i\xi\sqrt n \gamma} \lambda_{-{i\xi\over \sqrt n}}^n-e ^{-{\varrho^2\xi^2\over 2}}  \Big|\lesssim {1\over \sqrt n}e^{-{\varrho^2\xi^2\over 4}}.
\end{equation}

From the analyticity of $\xi\mapsto \oN_{i\xi}$ (cf. Proposition \ref{prop:spectral-decomp}), Lemma \ref{lemma:norm-Phi} and the fact that $\alpha A\leq 1/6$, one has 
$$\Big\|\big(\oN_{-{i\xi\over \sqrt n}}-\oN_0 \big)  (\Phi_{n,\xi} +\Phi_{n}^{\star}  ) \Big\|_\infty \lesssim {|\xi| \over \sqrt n} \norm{ \Phi_{n,\xi} +\Phi_{n}^{\star}  }_{\Cc^\alpha}\lesssim {|\xi| \over \sqrt n}n^{\alpha A} \leq  {|\xi| \over \sqrt n}\sqrt[6] n. $$
  Hence, using \eqref{eq:case1-lambda}, we get
$$\int_{\sqrt[6] n<|\xi|\leq \xi_0\sqrt n} {1\over |\xi|}\cdot \big| \Gamma_1(\xi) \big|  \,\diff \xi \lesssim \int_{- \infty}^{\infty} {1\over \sqrt[6] n} \cdot e^{-{\varrho^2\xi^2\over 3}} {|\xi| \over \sqrt n}\sqrt[6] n  \,\diff \xi  \lesssim   { 1\over \sqrt n}.$$

Observe that $ \big| \Phi_{n,\xi} +\Phi_{n}^{\star} \big| \leq \Phi_{n,0} +\Phi_{n}^{\star} = 1$, so  $\big|\oN_0  (\Phi_{n,\xi} +\Phi_{n}^{\star}  )\big| \leq 1$. Therefore, using \eqref{eq:case1-lambda}, we obtain
$$\int_{\sqrt[6] n<|\xi|\leq \xi_0\sqrt n} {1\over |\xi|}\cdot \big| \Gamma_2(\xi) \big|  \,\diff \xi \lesssim \int_{-\xi_0\sqrt n}^{\xi_0\sqrt n} {1\over \sqrt[6] n} \cdot{1\over \sqrt n} e^{-{\varrho^2\xi^2\over 4}}   \,\diff \xi  \lesssim { 1\over \sqrt n}.$$

The bound for $\Lambda_1$ follows in this case.
\vskip 5pt

\noindent\textbf{Case 2.} $|\xi|\leq \sqrt[6] n$.  In this case,  Lemma \ref{lemma:lambda-estimates} gives that
\begin{equation} \label{eq:case2-lambda}
\big|\lambda_{-{i\xi\over \sqrt n}}^n\big|\leq e^{-{\varrho^2\xi^2\over 3}} \quad \text{and}\quad \Big| e^{i\xi\sqrt n \gamma} \lambda_{-{i\xi\over \sqrt n}}^n-e ^{-{\varrho^2\xi^2\over 2}}  \Big|\lesssim {1\over \sqrt n}|\xi|^3e^{-{\varrho^2\xi^2\over 2}}.
\end{equation}

From \eqref{eq:psi_xi+psi_T} it follows that $\norm{\Phi_{n,\xi} +\Phi_{n}^{\star}}_{\Cc^\alpha}$ is bounded by
$$ 1 + \sum_{0\leq k\leq A\log n} \big|   e^{i \xi{k \over \sqrt n}}-1  \big|\cdot \norm{\chi_k}_{\Cc^\alpha}\lesssim 1 +  \sum_{0\leq k\leq A\log n}|\xi|{k \over \sqrt n} e^{\alpha k}  \lesssim 1 +  {\sqrt[6] n (\log n)^2 n^{\alpha A}\over \sqrt n} \lesssim 1 , $$
where in the last step we have used that $\alpha A\leq 1/6$. It follows from the analyticity of $\xi\mapsto\oN_{i\xi}$ that
$$  \Big\|\big(\oN_{-{i\xi\over \sqrt n}} -\oN_0\big)(\Phi_{n,\xi} +\Phi_{n}^{\star}  )  \Big\|_\infty  \lesssim {|\xi|\over \sqrt n}  .$$
We conclude, using \eqref{eq:case2-lambda}, that 
$$\int_{|\xi|\leq\sqrt[6] n} {1\over |\xi|}\cdot \big| \Gamma_1(\xi) \big|  \,\diff \xi \lesssim \int_{|\xi|\leq\sqrt[6] n} {1\over |\xi|} \cdot  e^{-{\varrho^2\xi^2\over 3}} {|\xi|\over \sqrt n}   \,\diff \xi \lesssim {1 \over \sqrt n}. $$

For $\Gamma_2$, using that $\big|\oN_0  (\Phi_{n,\xi} +\Phi_{n}^{\star}  )\big| \leq 1$ as before together with  \eqref{eq:case2-lambda}, gives
$$\int_{|\xi|\leq\sqrt[6] n} {1\over |\xi|}\cdot \big| \Gamma_2(\xi) \big|  \,\diff \xi \lesssim \int_{|\xi|\leq\sqrt[6] n} {1\over |\xi|}\cdot   {1\over \sqrt n}|\xi|^3e^{-{\varrho^2\xi^2\over 2}}   \,\diff \xi\lesssim {1 \over \sqrt n}.$$

Together with Case 1, we deduce that 
$$\Big|\int_{-\xi_0\sqrt n}^{\xi_0\sqrt n} {\Lambda_1(\xi) \over \xi}    \,\diff \xi\Big|  \lesssim {1 \over \sqrt n},$$
which ends the proof of the lemma.
\end{proof}

\begin{lemma} \label{lemma:theta-2-bound}
We have
$$\sup_{|b|\leq  \sqrt n}    \Big|  \int_{-\xi_0\sqrt n}^{\xi_0\sqrt n} {\Theta_b^{(2)}(\xi) \over \xi}   \,\diff \xi  \Big| \lesssim { 1\over \sqrt n}.$$
\end{lemma}

\begin{proof}
Recall that $\chi_k$ is bounded by $1$ and is supported by $ \Tc_k \subset \B(H_y,e^{-k+1})$. Therefore,
$$\oN_0 \chi_k = \int_{\P^{d-1}} \chi_k \, \diff \nu \leq \nu \big( \B(H_y,e^{-k+1}) \big)  \lesssim e^{-k\eta},$$
where in the last step we have used Proposition \ref{prop:regularity}.

Recall that  $\widehat h(\xi) =  e^{-{\varrho^2 \xi^2\over 2}}$ and $\widetilde h_n(\xi)=  e^{-{\varrho^2 \xi^2\over 2}} \oN_0 (\Phi_{n,\xi} +  \Phi_{n}^{\star})$. Using \eqref{eq:psi_xi+psi_T} and $\oN_0 \mathbf 1 = \mathbf 1$, we get
\begin{align*}
\Theta_b^{(2)}(\xi) &= e^{ib \xi} e^{-{\varrho^2 \xi^2\over 2}} \big( \oN_0 (\Phi_{n,\xi} +  \Phi_{n}^{\star}) - 1 \big) \widehat{\vartheta_{\delta_n}}(\xi) \\ &=  e^{ib \xi} e^{-{\varrho^2 \xi^2\over 2}} \sum_{0\leq k\leq A\log n} \big(  e^{i \xi{k \over \sqrt n}}-1 \big) \big( \oN_0 \chi_k \big) \cdot \widehat{\vartheta_{\delta_n}}(\xi).
\end{align*}
As $\norm{\widehat{\vartheta_{\delta_n}}}_\infty \leq 1$, we obtain 
\begin{align*}
\big|\Theta_b^{(2)}(\xi)\big| \leq  e^{-{\varrho^2 \xi^2\over 2}} \sum_{0\leq k\leq A\log n} \big|  e^{i \xi{k \over \sqrt n}}-1\big| \big( \oN_0\chi_k \big)    \lesssim e^{-{\varrho^2 \xi^2\over 2}}\sum_{k \geq 0} | \xi| {k \over \sqrt n} e^{-k\eta}\lesssim e^{-{\varrho^2 \xi^2\over 2}} {|\xi|\over \sqrt n},
\end{align*}
where the constants involved do not depend on $b$. Therefore,
$$   \int_{-\xi_0\sqrt n}^{\xi_0\sqrt n} \Big|{\Theta_b^{(2)}(\xi)\over \xi} \Big|  \,\diff \xi   \lesssim  \int_{-\xi_0\sqrt n}^{\xi_0\sqrt n}  {1\over |\xi|} \cdot e^{-{\varrho^2 \xi^2\over 2}} {|\xi|\over \sqrt n}  \,\diff \xi \lesssim {1 \over \sqrt n},$$
thus proving the lemma.
\end{proof}

Gathering the above estimates, we can finish the proof.

\begin{proof}[End of proof of Theorem \ref{thm:BE-coeff}] 
Recall that our goal is to prove \eqref{goal-1-varphi}. Estimate \eqref{eq:BE-main-estimate} together with Lemmas \ref{lemma:theta-1-bound} and \ref{lemma:theta-2-bound} give that
$\big|F_n(b)-H(b)\big| \leq C'/ \sqrt n$ for all $b \in \R$, where $C' > 0$ is a constant. Recall that $H(b): =\frac{1}{\sqrt{2 \pi}  \,\varrho} \int_{-\infty}^b e^{-\frac{s^2}{2 \varrho^2}} \, \diff s$. Coupling the last estimate with Lemma \ref{lemma:Ln-Fn} and the easy fact that $\sup_{b \in \R} \big|H(b)-H(b \pm 1/ \sqrt n)\big| \lesssim 1/ \sqrt n$ gives that $ \big|\oL_n(b)-H(b)\big| \leq C'' /  \sqrt n$ for some constant $C'' > 0$. Therefore,  \eqref{goal-1-varphi} holds. Observe that all of our estimates are uniform in $x \in \P^{d-1}$ and $y\in (\P^{d-1})^*$. The proof of the theorem is complete.
\end{proof}

\section{Local limit theorem for coefficients} \label{sec:LLT}

This section is devoted to the proof of  Theorem \ref{thm:LLT-coeff}.  As in the previous section, we fix $x =[v]\in \P^{d-1}$ and $y=[f]\in (\P^{d-1})^*$.   Fix also  $-\infty<a<b<\infty$ and define
\begin{equation*}
\oA_n(t) := \sqrt{n}  \, \mathbf P \Big(  t+ \log{ |\lp f, S_n v \rp | \over \norm{f} \norm{v}} - n \gamma\in [ a, b] \Big) .
\end{equation*}
Our goal is to prove that
\begin{equation} \label{eq:LLT-main-limit}
\lim_{n\to \infty}\sup_{t\in\R} \Big|\, \oA_n(t)  -   e^{-\frac{t^2}{2 \varrho^2 n}} {b-a\over \sqrt{2 \pi}\,\varrho} \, \Big| =0.
\end{equation}

Our strategy is similar to the one employed in the proof of Theorem \ref{thm:BE-coeff}, that is, to replace $\log{ |\lp f, S_n v \rp | \over \norm{f} \norm{v}} $ by $\sigma(S_n,x) +\log d(S_n x, H_y)$ (see  \eqref{eq:coeff-split}) and use the perturbed Markov operator and large deviation estimates to handle $\sigma(S_n,x)$  and $\log d(S_n x, H_y)$. Here, we are dealing with ``local'' probabilities for $\sigma(S_n,x)$, so the analysis is more involved. In particular, we need to use finer approximation results, such as the one in Lemma \ref{lemma:conv-fourier-approx} and properties of the operator $\oP_{i\xi}$ for large values of $\xi \in \R$, as in Proposition  \ref{prop:spec-Pxi}.

\medskip

 Let $0<\zeta \leq 1$ be a constant. For integers $k \geq 0$ introduce
\begin{align*}
\Tc_k^\zeta := \big\{ w \in \P^{d-1} :\, e^{-(k+1)\zeta} < d(w,H_y) < e^{-(k-1)\zeta} \big\} = \B(H_y,e^{-(k-1) \zeta}) \setminus \overline{\B(H_y,e^{-(k+1)\zeta})}.
\end{align*}

We have the following version of Lemma \ref{lemma:partition-of-unity}. We'll use the same notation to denote slightly different functions. This shouldn't cause confusion. The functions from Lemma \ref{lemma:partition-of-unity} correspond to the particular case $\zeta =1$.

\begin{lemma} \label{lemma:partition-of-unity-2}
Let $0<\zeta \leq 1$. There exist non-negative smooth functions $\chi_k$ on $\P^{d-1}$, $k \geq 0$, such that
\begin{enumerate}
\item $\chi_k$ is supported by $\Tc_k^\zeta$;
\item If $w \in \P^{d-1} \setminus H_y$,  then  $\chi_k(w) \neq 0$ for at most two values of $k$;
\item $\sum_{k\geq 0}  \chi_k=1$ on $\P^{d-1} \setminus H_y$; 
\item $\norm{\chi_k}_{\Cc^1}\leq 12  \zeta^{-1} e^{k\zeta}$.
\end{enumerate}
\end{lemma}

\begin{proof}
Let $\widetilde \chi_k$  be as in Lemma \ref{lemma:partition-of-unity} and set $\chi_k(w):= \widetilde \chi_k \big(\zeta^{-1} \log d(w,H_y) \big)$.  Since the function $\Psi(w) := \log d(w,H_y)$ satisfies $\norm{\Psi|_{\Tc_k^\zeta }}_{\Cc^1} \leq e^{(k+1) \zeta} \leq 3 e^{k \zeta}$,  it follows that $\chi_k$ satisfies (1)--(4).
\end{proof}

 We will prove \eqref{eq:LLT-main-limit} by dealing separately with the upper and lower limit.
 
 \subsection{Upper bound} \label{subsec:LLT-upper-limit}
 
 The upper bound in the limit \eqref{eq:LLT-main-limit} is handled by the following proposition.
 
 \begin{proposition} \label{prop:LLT-upper-limit}
 Let $\oA_n(t)$ be as above. Then, $$\limsup_{n\to \infty}\sup_{t\in\R}\bigg( \oA_n(t)  -   e^{-\frac{t^2}{2 \varrho^2 n}} {b-a\over \sqrt{2 \pi}\,\varrho} \bigg) \leq 0.$$
 \end{proposition}

	Let $0< \zeta \leq 1$ be a small constant and define $\psi:\R\to\R_{\geq 0}$  by
	 \begin{align*}
	\psi (u) := &
	\begin{cases}
	  u/\zeta -(a-2\zeta)/\zeta     & \text{for}     \quad     u\in[a-2\zeta,a-\zeta]              \\
	1             & \text{for} \quad        u\in[a-\zeta,b+\zeta]          \\
     -u/\zeta +(b+2\zeta)/\zeta       & \text{for}  \quad           u\in [b+\zeta,b+2\zeta]
     \\
	0             & \text{for} \quad        u\in \R \setminus [a-2\zeta,b+2\zeta].
	\end{cases}
	\end{align*}
Notice that  $\psi$ is Lipschitz and piecewise affine.  Moreover, $0 \leq \psi \leq 1$, its support  is contained in $[a-2\zeta,b+2\zeta] \subset [a-2,b+2]$ and  $\int_{\R} \psi (u)\,\diff u = b-a+3\zeta$.	
	
	For $t\in\R$ and $k\in\N$, consider the translations $$\psi_{t,k}(u):=\psi(u+t-k\zeta).$$
	Observe that, for fixed $t,u\in\R$, we have that $\psi_{t,k}(u)\neq 0$ for only finitely many $k$'s. By construction, we have $\psi_{t,k} \geq \mathbf 1_{[a-t+(k-1)\zeta, b-t+(k+1)\zeta]}$.
	\medskip
	
Let $B>0$ be a large constant. Arguing as in Section \ref{sec:BE}, we obtain that,  for $n$ large,
	$$\mu^{*n} \big\{g\in G:\,d(gx,H_y)\leq e n^{-B}       \big\}\leq e^{2c} \, n^{-c  B}, $$
	for some constant $c> 0$ independent of $n$ and $B$. Taking $B$ large enough allows us to assume that $n^{-c B } \leq 1 / n$  for all $n \geq 1$.
	
	In order to simplify the notation,  consider the linear functional
	\begin{equation} \label{eq:En-def}
\oE_n\big(\Psi\big):= \sqrt n \, \mathbf E\Big( \Psi\big( \sigma(S_n,x)-n\gamma ,S_n x \big)  \Big),
\end{equation}
where $\Psi$ is a function of $(u,w) \in \R \times \P^{d-1}$.

	For $w \in \P^{d-1}$, set 
	\begin{equation} \label{eq:Phi-star-def}
		\Phi_n^\star (w):= 1 -  \sum_{0\leq k\leq B\zeta^{-1}\log n}  \chi_k(w),
	\end{equation}	
	where $\chi_k$ are the functions in Lemma \ref{lemma:partition-of-unity-2}.  We observe that we use the same notation as in Section \ref{sec:BE} to denote  slightly different functions. We recover the functions from last section by taking $\zeta = 1$ and $B=A$.  This shouldn't cause any confusion.
	
	Set
	$$\oB_n(t):=  \sum_{0\leq k\leq B\zeta^{-1}\log n} \oE_n\big(\psi_{t,k}\cdot\chi_k  \big)+   \oE_n\big(\psi_{t,0}\cdot \Phi_n^\star \big).$$
	
	\begin{lemma} \label{lemma:llt-ineq-1}
	There exists a constant $C_1>0$,  independent of $n$ and $\zeta$,  such that,  for all $t \in \R$,   $$\oA_n(t) \leq \oB_n(t) + C_1 / \sqrt n.$$
	\end{lemma}
	
\begin{proof}
Using the decomposition \eqref{eq:coeff-split} and the fact that $\mathbf P \big( d(S_n x,H_y)\leq  n^{-B} \big) \lesssim 1 / n$, we obtain
\begin{align*}
	\oA_n(t) \leq \sqrt{n}  \, \mathbf E \Big(  \mathbf 1_{t+ \sigma(S_n,x) +\log d(S_n x, H_y) - n \gamma\in [a, b]}  \mathbf 1_{\log d(S_n x, H_y)\geq -B\log n}  \Big)    + O \Big( {1 \over \sqrt n} \Big) .
	\end{align*}
	Observe that,  when $S_n x \in\supp(\chi_k)$, we have $-(k+1) \zeta \leq \log d(S_n x,H_y) \leq - (k-1) \zeta$,  so
	$$  \mathbf 1_{t+ \sigma(S_n,x) +\log d(S_n x, H_y) - n \gamma\in [a, b]}\leq \mathbf 1_{\sigma(S_n,x)  - n \gamma\in [a-t+(k-1)\zeta, b-t+(k+1)\zeta]} \leq \psi_{t,k}\big(\sigma(S_n,x)  - n \gamma\big).      $$
	Using that $\mathbf 1_{\log d(w, H_y)\geq -B\log n} \leq \sum_{0\leq k\leq B\zeta^{-1}\log n + 1} \chi_k(w)$ and taking the expectation, it follows that
	\begin{align*}
	&\E \Big(\mathbf 1_{t+ \sigma(S_n,x) +\log d(S_n x, H_y) - n \gamma\in [a, b]}  \mathbf 1_{\log d(S_n x, H_y)\geq -B\log n} \Big) \\
	&\leq  \sum_{0\leq k\leq B\zeta^{-1}\log n + 1} \E \Big( \psi_{t,k}\big(\sigma(S_n,x)-n\gamma\big)\chi_k(S_n x) \Big)    \\ &\leq  \sum_{0\leq k\leq B\zeta^{-1}\log n} \E \Big(\psi_{t,k}\big(\sigma(S_n,x)-n\gamma\big)\chi_k(S_n x) \Big) +   \E\Big(\chi_{k_0}(S_n x) \Big),
	\end{align*}
where $k_0:= \lfloor B \zeta^{-1} \log n \rfloor + 1$.  

From the fact that $\chi_{k_0} \leq \mathbf 1_{\B(H_y,e^{-(k_0-1)\zeta})}$, we see that the last term above is bounded by $\mathbf P \big( d(S_n x,H_y)\leq  e n^{-B} \big) \lesssim 1 / n$. Hence, there is a constant $C_1 >0$ such that
	\begin{equation*}
	\oA_n(t) \leq	 \sum_{0\leq k\leq B\zeta^{-1}\log n} \oE_n\big(\psi_{t,k}\cdot\chi_k  \big)+{C_1 \over \sqrt n} \leq \oB_n(t)  +{C_1 \over \sqrt n},
	\end{equation*}
	proving the lemma.
\end{proof}

	By Lemma \ref{lemma:conv-fourier-approx}, for every $0<\delta<1$, there exists a smooth function $\psi^+_{\delta}$  such that $\widehat {\psi^+_{\delta}}$ has support in $[-\delta^{-2},\delta^{-2}]$,  $$\psi\leq \psi^+_\delta,\quad \lim_{\delta\to 0} \psi^+_{\delta} =\psi    \quad \text{and} \quad  \lim_{\delta\to 0} \big \|\psi^+_{\delta} -\psi \big \|_{L^1} = 0.$$ 
	Moreover,  $\norm{\psi_{\delta}^+}_\infty$, $\norm{\psi_{\delta}^+}_{L^1}$ and $\|\widehat{\psi^+_{\delta}}\|_{\Cc^1}$ are bounded by a constant independent of $\delta$ and $\zeta$ since the support of $\psi$ is contained in $[a-2,b+2]$. 
	
	As above, for $t\in\R$ and $k\in\N$, we consider the translations
	$$ \psi_{t,k}^+(u):=\psi_{\delta}^+(u+t-k\zeta)  . $$
	We omit the dependence on $\delta$ in order to ease the notation. Define also
	\begin{equation} \label{eq:R-def}
	\oR_n(t):= \sum_{0\leq k\leq B\zeta^{-1}\log n } \oE_n\big(\psi_{t,k}^+\cdot\chi_k  \big)+  \oE_n\big(\psi_{t,0}^+\cdot \Phi_n^\star  \big).	
\end{equation}

	Clearly, we have $\oB_n(t)\leq \oR_n(t)$. From the definition of $\oE_n$, Fourier inversion formula and Fubini's theorem , we have
	\begin{align*}
	\oE_n\big(\psi_{t,k}^+\cdot\chi_k \big)&=\sqrt n \, \int_{G} \psi_{\delta}^+\big(\sigma(g,x)-n\gamma+t-k\zeta\big) \cdot \chi_k(gx) \,\diff \mu^{*n}(g)\\
	&={\sqrt n\over 2\pi}\int_{G} \int_{-\infty}^\infty \widehat{\psi_{\delta}^+}(\xi) e^{i\xi(\sigma(g,x)-n\gamma+t-k\zeta )} \cdot \chi_k(gx) \,\diff \xi\diff\mu^{*n}(g)\\
	&={\sqrt n\over 2\pi}\int_{-\infty}^\infty  \widehat{\psi_{\delta}^+}(\xi) e^{i\xi(t-k\zeta)}\cdot e^{-i\xi n\gamma}\oP^n_{i\xi}\chi_k(x) \,\diff \xi,
	\end{align*}
	where in the last step we have used \eqref{eq:markov-op-iterate}.
	
	Recall that $\supp\big( \widehat{\psi_{\delta}^+} \big)\subset [-\delta^{-2},\delta^{-2}]$. So, after the change of variables $\xi \mapsto \xi / \sqrt n$, the above identity becomes 
	$$	\oE_n\big(\psi_{t,k}^+\cdot\chi_k \big) ={1\over 2\pi}\int_{-\delta^{-2} \sqrt n}^{\delta^{-2} \sqrt n} \widehat{\psi_{\delta}^+}\Big({\xi\over \sqrt n}\Big) e^{i\xi{t-k\zeta\over \sqrt n}}\cdot e^{-i\xi \sqrt n\gamma}\oP^n_{{i\xi\over \sqrt n}}\chi_k(x) \,\diff \xi.$$
	
	A similar computation yields $$\oE_n\big(\psi_{t,0}^+ \cdot \Phi_n^\star \big) = {1\over 2\pi}\int_{-\delta^{-2} \sqrt n}^{\delta^{-2} \sqrt n} \widehat{\psi_{\delta}^+}\Big({\xi\over \sqrt n}\Big) e^{i\xi{t \over \sqrt n}}\cdot e^{-i\xi \sqrt n\gamma}\oP^n_{{i\xi\over \sqrt n}}\Phi_{n}^{\star} (x) \,\diff \xi.$$

	Define
	\begin{equation} \label{eq:Phi-xi-def}
	\Phi_{n,\xi} (w):= \sum_{0\leq k\leq B\zeta^{-1}\log n}  e^{-i \xi{k\zeta\over \sqrt n}}\chi_k(w).
	\end{equation}
	We use again the same notation as in Section \ref{sec:BE} to denote a slightly different function. The difference here is the factor $\zeta$ and the sign before $i\xi$. Using this notation and the above computations, \eqref{eq:R-def} becomes
\begin{equation} \label{eq:R-formula} 
\oR_n(t)= {1\over 2\pi}\int_{-\delta^{-2} \sqrt n}^{\delta^{-2} \sqrt n} \widehat{\psi_{\delta}^+}\Big({\xi\over \sqrt n}\Big) e^{i\xi{t\over \sqrt n}}\cdot e^{-i\sqrt n\gamma}\oP_{{i\xi\over \sqrt n}}^n(\Phi_{n,\xi} +\Phi_{n}^{\star}  )(x) \,\diff \xi.
\end{equation}

Fix $\alpha>0$ such that $$\alpha B\leq 1/6 \quad \text{ and } \quad \alpha \leq \alpha_0,$$ where $0<\alpha_0<1$ is the exponent appearing in Theorem \ref{thm:spectral-gap}. Then, all the results of Subsection \ref{subsec:markov-op} apply to the operators $\xi\mapsto\oP_{i\xi}$ acting on $\Cc^\alpha(\P^{d-1})$.

	The next lemma can be proved in the same way as Lemma \ref{lemma:norm-Phi}.
	
	\begin{lemma}  \label{lemma:norm-Phi-2}
Let $0 < \zeta \leq 1$, $\Phi_{n,\xi}, \Phi_{n}^{\star}$ and $\alpha > 0$ be as above. Then, 
\begin{equation} \label{eq:psi_xi+psi_T-2}
\Phi_{n,\xi} + \Phi_{n}^{\star} = \mathbf 1 +  \sum_{0\leq k\leq B\zeta^{-1}\log n}  \big(e^{-i \xi{k\zeta\over \sqrt n}} - 1 \big) \chi_k
\end{equation}
and there is a constant $C_\zeta>0$ independent of $n$ and $\xi$ such that
\begin{equation*}  \label{eq:norm-Phi-2}
	\norm{\Phi_{n,\xi} }_{\Cc^\alpha}\leq C_\zeta n^{\alpha B} \quad\text{and}\quad  \norm{\Phi_{n}^{\star}  }_{\Cc^\alpha}\leq C_\zeta n^{\alpha B}.
\end{equation*}
Moreover, $\Phi_{n}^{\star}  $ is supported by $\big\{w:\,\log d(w,H_y)\leq -B\log n + 1\big\}$.
\end{lemma}
	
 	Define  
 	\begin{equation} \label{eq:S-def} 
 	\oS_n(t):={1\over 2\pi}\widehat{\psi_{\delta}^+}(0)  \int_{-\infty}^\infty e^{i\xi{t\over \sqrt n}} e^{-{ \varrho^2\xi^2 \over 2}}  \,\diff \xi ={1\over \sqrt{2\pi} \, \varrho} e^{-{t^2\over2\varrho^2 n}}\int_{\R} \psi_{\delta}^+ (u)\,\diff u,
 	\end{equation}
	where in the second equality we have used the fact that the inverse Fourier transform of $e^{-{ \varrho^2\xi^2 \over 2}}$ is $ {1\over \sqrt{2\pi} \, \varrho}  e^{-{t^2\over2\varrho^2}}$.

   \begin{lemma}\label{lemma-R-S}
	Fix $0< \delta < 1$ and $0<\zeta \leq 1$. Then, there exists a constant $C_{\zeta,\delta}>0$ such that, for all $n \geq 1$,
	$$\sup_{t\in\R}\big|\oR_n(t)-\oS_n(t) \big| \leq {C_{\zeta,\delta}\over \sqrt[3] n}.$$
	\end{lemma}
	\begin{proof}
	Let $\xi_0>0$ be the constant in Lemma \ref{lemma:lambda-estimates}. In particular, the  decomposition of $\oP_z$ in Proposition \ref{prop:spectral-decomp} holds for $|z| \leq \xi_0$. Using that decomposition, \eqref{eq:R-formula} and \eqref{eq:S-def}, we can write $$\oR_n(t)-\oS_n(t) = \Lambda_n^1(t) + \Lambda_n^2(t)  + \Lambda_n^3(t)  + \Lambda_n^4(t)  + \Lambda_n^5(t),$$ where
	$$\Lambda_n^1(t):={1\over 2\pi}\int_{-\xi_0 \sqrt n}^{\xi_0 \sqrt n} e^{i\xi{t\over \sqrt n}} \Big[ \widehat{\psi_{\delta}^+}\Big({\xi\over \sqrt n}\Big) e^{-i\sqrt n\gamma}\lambda_{{i\xi\over \sqrt n}}^n\oN_0(\Phi_{n,\xi} +\Phi_{n}^{\star}  )-\widehat{\psi_{\delta}^+}(0)    e^{-{ \varrho^2\xi^2 \over 2}}\Big]  \,\diff \xi ,$$
	$$\Lambda_n^2(t):= {1\over 2\pi}\int_{-\xi_0 \sqrt n}^{\xi_0 \sqrt n} e^{i\xi{t\over \sqrt n}}\Big[ \widehat{\psi_{\delta}^+}\Big({\xi\over \sqrt n}\Big) e^{-i\sqrt n\gamma}\lambda_{{i\xi\over \sqrt n}}^n\big(\oN_{{i\xi\over \sqrt n}}-\oN_0\big)(\Phi_{n,\xi} +\Phi_{n}^{\star}  ) (x) \Big]  \,\diff \xi , $$
	$$\Lambda_n^3(t):=  {1\over 2\pi}\int_{-\xi_0 \sqrt n}^{\xi_0 \sqrt n} e^{i\xi{t\over \sqrt n}}  \widehat{\psi_{\delta}^+}\Big({\xi\over \sqrt n}\Big)  e^{-i\sqrt n\gamma} \oQ_{{i\xi\over \sqrt n}}^n(\Phi_{n,\xi} +\Phi_{n}^{\star}  )(x)   \,\diff \xi,   $$
	$$ \Lambda_n^4(t):=  {1\over 2\pi}\int_{\xi_0\sqrt n \leq|\xi|\leq\delta^{-2} \sqrt n}e^{i\xi{t\over \sqrt n}} \widehat{\psi_{\delta}^+}\Big({\xi\over \sqrt n}\Big) e^{-i\sqrt n\gamma}\oP_{{i\xi\over \sqrt n}}^n(\Phi_{n,\xi} +\Phi_{n}^{\star}  )(x) \,\diff \xi   $$
	and 
	$$ \Lambda_n^5(t):=  - {1\over 2\pi}\widehat{\psi_{\delta}^+}(0)  \int_{|\xi|\geq \xi_0 \sqrt n} e^{i\xi{t\over \sqrt n}}  e^{-{ \varrho^2\xi^2 \over 2}} \,\diff \xi. $$
	\medskip
	
	We will bound each $\Lambda_n^j$, $j=1,\ldots, 5$,  separately. We will use that $$\norm{\Phi_{n,\xi} +\Phi_{n}^{\star}  }_{\Cc^\alpha}\leq 2C_\zeta n^{\alpha B} \leq 2C_\zeta n^{1 / 6}$$ for every $\xi$, after Lemma \ref{lemma:norm-Phi-2} and the choice of $\alpha$ and $B$.
	
	In order to bound $\Lambda_n^2$, we have, using the analyticity of $\xi\mapsto \oN_{i\xi}$, that
	$$ \Big\| \big(\oN_{{i\xi\over \sqrt n}}-\oN_0\big)(\Phi_{n,\xi} +\Phi_{n}^{\star}  ) \Big\|_\infty \lesssim  {|\xi|\over \sqrt n} \norm{\Phi_{n,\xi} +\Phi_{n}^{\star}  }_{\Cc^\alpha}\leq  { 2C_\zeta|\xi|\over \sqrt [3] n}.$$
	Recall, from Lemma \ref{lemma:lambda-estimates}, that $\big|\lambda_{{i\xi\over \sqrt n}}^n\big|\leq e^{-{\varrho^2\xi^2\over 3}}$ for $|\xi|\leq \xi_0 \sqrt n$. 	Since $\|\widehat{\psi^\pm_{\delta}}\|_{\Cc^1}$ is bounded uniformly in $\delta$ and $\zeta$, we get 
	$$\sup_{t\in\R} \big|\Lambda_n^2(t)\big|\lesssim  \int_{-\infty}^{\infty} e^{-{\varrho^2\xi^2\over 3}}  {2C_\zeta|\xi|\over \sqrt [3] n} \,\diff \xi \lesssim {C_\zeta\over \sqrt[3] n}. $$
	
	For  $\Lambda_n^3$, we use that $\norm{\oQ^n_z}_{\Cc^\alpha} \leq c \beta^n$ for $|z| \leq \xi_0$, where $c>0$ and $0<\beta<1$ are constants, see Proposition \ref{prop:spectral-decomp}. Therefore, for $|\xi|\leq \xi_0\sqrt n$,
	$$\Big\| \oQ_{{i\xi\over \sqrt n}}^n(\Phi_{n,\xi} +\Phi_{n}^{\star}  ) \Big\|_\infty \lesssim \beta^n\norm{\Phi_{n,\xi} +\Phi_{n}^{\star}  }_{\Cc^\alpha} \leq 2 C_\zeta \beta^n  \sqrt[6] n,    $$
which gives
	$$\sup_{t\in\R} \big|\Lambda_n^3(t)\big|\lesssim \int_{-\xi_0 \sqrt n}^{\xi_0 \sqrt n}  2 C_\zeta \beta^n  \sqrt[6] n \,\diff \xi = 4\xi_0 C_\zeta \sqrt n  \beta^n  \sqrt[6] n\lesssim \frac{C_\zeta}{\sqrt n}.$$

	In order to bound  $\Lambda_n^4$,  we use that, after Proposition \ref{prop:spec-Pxi}, there are constants $C_\delta>0$ and $0<\rho_\delta<1$ such that $\norm{\oP^n_{i\xi}}_{\Cc^\alpha}\leq C_\delta \rho_\delta^n$ for all $\xi_0\leq |\xi|\leq \delta^{-2}$ and $n \geq 1$. Therefore,
	$$\sup_{t\in\R}  \big|\Lambda_n^4(t)\big|\lesssim \int_{\xi_0\sqrt n \leq|\xi|\leq\delta^{-2} \sqrt n} C_\delta \rho_\delta^n \sqrt[6] n\,\diff \xi\leq 2\delta^{-2}\sqrt n C_\delta \rho_\delta^n C_\zeta \sqrt[6] n\lesssim \frac{C_{\zeta,\delta}'}{\sqrt n},$$
	for some constant $C_{\zeta,\delta}'>0$.
	
	The modulus of the term $\Lambda_n^5$ is clearly $\lesssim 1/\sqrt n$, so it only remains to estimate $\Lambda_n^1$. For every $t \in \R$, we have $$\big| \Lambda_n^1(t) \big|\leq \Gamma_n^1+\Gamma_n^2+\Gamma_n^3,$$ where 
	$$\Gamma_n^1:= {1\over 2\pi}\int_{-\xi_0 \sqrt n}^{\xi_0 \sqrt n} \Big| \widehat{\psi_{\delta}^+}\Big({\xi\over \sqrt n}\Big) \Big| \, \big|\lambda_{{i\xi\over \sqrt n}}^n \big| \cdot\Big| \oN_0(\Phi_{n,\xi} +\Phi_{n}^{\star}  )- 1 \Big|  \,\diff \xi , $$
	$$\Gamma_n^2:= {1\over 2\pi}\int_{-\xi_0 \sqrt n}^{\xi_0 \sqrt n} \big|\lambda_{{i\xi\over \sqrt n}}^n\big|\cdot \Big| \widehat{\psi_{\delta}^+}\Big({\xi\over \sqrt n}\Big) -\widehat{\psi_{\delta}^+}(0) \Big|  \,\diff \xi $$
	and
	$$\Gamma_n^3:= {1\over 2\pi}\int_{-\xi_0 \sqrt n}^{\xi_0 \sqrt n} \big| \widehat{\psi_{\delta}^+}(0) \big| \cdot\Big| e^{-i\sqrt n\gamma}\lambda_{{i\xi\over \sqrt n}}^n-    e^{-{ \varrho^2\xi^2 \over 2}}\Big|  \,\diff \xi. $$
	
	Recall that $\chi_k$ is bounded by $1$ and is supported by $ \Tc_k^\zeta \subset \B(H_y,e^{-(k-1)\zeta})$. Therefore,
$$\oN_0 \chi_k = \int_{\P^{d-1}} \chi_k \, \diff \nu \leq \nu \big(\B(H_y,e^{-(k-1) \zeta})\big)  \lesssim e^{-k\zeta \eta},$$
where in the last step we have used Proposition \ref{prop:regularity}.
	
	Using \eqref{eq:psi_xi+psi_T-2},  we get 
	\begin{align*}
	\Big| \oN_0(\Phi_{n,\xi} +\Phi_{n}^{\star}  )- 1 \Big| &= \Big| \oN_0(\Phi_{n,\xi} +\Phi_{n}^{\star}  )-\oN_0 \mathbf 1\Big|\leq \sum_{0\leq k\leq B\zeta^{-1}\log n} \big| e^{-i \xi{k\zeta\over \sqrt n}}-1 \big|\oN_0 \chi_k \\
	&\lesssim \sum_{ k \geq 0} |\xi|{k\zeta\over \sqrt n}e^{-k \zeta\eta}\leq c_\zeta {|\xi|\over \sqrt n},
	\end{align*}
	for some constant $c_\zeta>0$ independent of $n$.
	
	Using that $\big|\lambda_{{i\xi\over \sqrt n}}^n\big|\leq e^{-{\varrho^2\xi^2\over 3}}$ for $|\xi|\leq \xi_0 \sqrt n$ (Lemma \ref{lemma:lambda-estimates}) and that  $\|\widehat{\psi^+_{\delta}}\|_{\Cc^1}$  is uniformly bounded,  we get that 
	$$ \Gamma_n^1\lesssim  \int_{-\xi_0 \sqrt n}^{\xi_0 \sqrt n}  \|\widehat{\psi^+_{\delta}}\|_{\infty} e^{-{\varrho^2\xi^2\over 3}} c_\zeta {|\xi|\over \sqrt n}  \,\diff \xi\lesssim {c_\zeta\over \sqrt n}$$
	and
	$$\Gamma_n^2\lesssim  \int_{-\xi_0 \sqrt n}^{\xi_0 \sqrt n} e^{-{\varrho^2\xi^2\over 3}} {|\xi|\over \sqrt n} \|\widehat{\psi^+_{\delta}}\|_{\Cc^1} \,\diff \xi\lesssim {1\over \sqrt n}.$$

	The bound $\Gamma_n^3\lesssim 1/\sqrt n$ follows by splitting the integral along the intervals $|\xi|\leq \sqrt[6] n$ and $\sqrt[6] n< |\xi| \leq \xi_0\sqrt n$ and using Lemma \ref{lemma:lambda-estimates}.
	
	 We conclude that $$\sup_{t\in\R} \big|\Lambda_n^1(t)\big|\lesssim \frac{c_\zeta'}{\sqrt n},$$ for some constant $c_\zeta'>0$ independent of $n$.
	
	\medskip
	
	Gathering the above estimates, we obtain $$\sup_{t\in\R}\big|\oR_n(t)-\oS_n(t) \big|\lesssim \frac{c_\zeta'}{\sqrt n} + {C_\zeta\over \sqrt[3] n} +  \frac{C_\zeta}{\sqrt n} +  \frac{C_{\zeta,\delta}'}{\sqrt n} + \frac{1}{\sqrt n}.$$
Hence,  the above quantity is bounded by  $C_{\zeta,\delta} / \sqrt[3] n$ for some constant $C_{\zeta,\delta} > 0$. This finishes the proof of the lemma.
		\end{proof}
	
	The above estimates are enough to obtain the desired upper bound.
	
	\begin{proof}[Proof of Proposition \ref{prop:LLT-upper-limit}]

	Fix $0<\delta<1$ and $0 < \zeta \leq 1$ as in  the beginning of this subsection. Lemmas \ref{lemma:llt-ineq-1} and \ref{lemma-R-S} and the fact that $\oB_n(t) \leq \oR_n(t)$ give that
	$$\oA_n(t) \leq \oS_n(t) + {C_{\zeta,\delta}\over \sqrt[3] n}+{C_1 \over \sqrt n} \quad \text{for all } \,\, t \in \R.$$	
	Recall, from \eqref{eq:S-def},  that $\oS_n(t) = {1\over \sqrt{2\pi} \, \varrho} e^{-{t^2\over2\varrho^2 n}}\int_{\R} \psi_{\delta}^+ (u)\,\diff u$ and $\int_{\R} \psi (u)\,\diff u = b-a+3\zeta$. Hence, for every fixed $n$ and $\zeta$,  
	$$ \Big|  \oS_n(t)- e^{-{t^2\over2\varrho^2 n}}{b-a+3\zeta\over \sqrt{2\pi} \, \varrho} \Big| \leq {1\over \sqrt{2\pi} \, \varrho}  \big\|   \psi_{\delta}^+ -\psi\big\|_{L^1}.$$
	
	We deduce that 
	$$ \oA_n(t) -   e^{-\frac{t^2}{2 \varrho^2 n}} {b-a\over \sqrt{2 \pi}\,\varrho} \leq  e^{-{t^2\over2\varrho^2 n}} {3\zeta\over \sqrt{2\pi} \, \varrho} +   {1\over \sqrt{2\pi} \, \varrho}  \big\|   \psi_{\delta}^+ -\psi\big\|_{L^1} +   {C_{\zeta,\delta}\over \sqrt[3] n}+{C_1 \over \sqrt n}, $$ so
	$$\limsup_{n\to \infty}  \sup_{t\in\R} \bigg( \oA_n(t) -   e^{-\frac{t^2}{2 \varrho^2 n}} {b-a\over \sqrt{2 \pi}\,\varrho} \bigg) \leq     {3\zeta\over \sqrt{2\pi} \, \varrho} + {1\over \sqrt{2\pi} \, \varrho}  \big\|   \psi_{\delta}^+ -\psi\big\|_{L^1}.$$
	
	Since $0<\delta<1$ and $0 < \zeta \leq 1$ are arbitrary and $\big\|   \psi_{\delta}^+ -\psi\big\|_{L^1}$ tends to zero as $\delta \to 0$, the proposition follows.
		\end{proof}

\subsection{Lower bound} \label{subsec:LLT-lower-limit}

 We now deal with the lower bound in the limit in \eqref{eq:LLT-main-limit}. 
 
 \begin{proposition} \label{prop:LLT-lower-limit}
 Let $\oA_n(t)$ be as above. Then, $$\liminf_{n\to \infty}\inf_{t\in\R}\bigg( \oA_n(t)  -   e^{-\frac{t^2}{2 \varrho^2 n}} {b-a\over \sqrt{2 \pi}\,\varrho} \bigg) \geq 0.$$
 \end{proposition}
 
 The argument is a variation of the one used in the proof of Proposition \ref{prop:LLT-upper-limit}, but the upper approximations used above will be replaced by analogous lower approximations. We now give the details.

Let $0< \zeta \leq 1$ be a small constant and define $\widetilde \psi:\R\to\R_{\geq 0}$ by
	 \begin{align*}
	\widetilde \psi (u)= &
	\begin{cases}
	u/\zeta -(a+\zeta)/\zeta     & \text{for}     \quad     u\in[a+\zeta,a+2\zeta]              \\
	1             & \text{for} \quad        u\in[a+2\zeta,b-2\zeta]          \\
	-u/\zeta +(b-\zeta)/\zeta       & \text{for}  \quad           u\in [b-2\zeta,b-\zeta]
\\
	0             & \text{for} \quad        u\in \R \setminus [a+\zeta,b-\zeta].
	\end{cases}
	\end{align*}
The function $\widetilde \psi$ is Lipschitz and piecewise affine. Moreover,  $0 \leq \widetilde \psi \leq 1$,  its support  is contained in $[a+\zeta,b-\zeta] \subset [a,b]$ and  $\int_{\R} \widetilde \psi (u)\,\diff u = b-a - 3\zeta$.

	For $t\in\R$ and $k\in\N$, consider the translations
	$$\widetilde\psi_{t,k}(u):=\widetilde\psi(u+t-k\zeta).$$

	Then, for fixed $t,u\in\R$, we have that $\widetilde \psi_{t,k}(u)\neq 0$ for only finitely many $k$'s and $\mathbf 1_{[a-t+(k+1)\zeta, b-t+(k-1)\zeta]} \geq \widetilde\psi_{t,k}$.
	
	Let $\chi_k$ be as in Lemma \ref{lemma:partition-of-unity-2} and $\Phi_n^\star$ be the function defined in \eqref{eq:Phi-star-def}. Set
	$$\widetilde \oB_n(t):=  \sum_{0\leq k\leq B\zeta^{-1}\log n} \oE_n\big(\widetilde \psi_{t,k}\cdot\chi_k  \big)+   \oE_n\big(\widetilde \psi_{t,0}\cdot \Phi_n^\star \big),$$
	where $\oE_n$ is defined in \eqref{eq:En-def}.
	
	\begin{lemma} \label{lemma:llt-ineq-2}
	There exists a constant $C_2>0$,  independent of $n$ and $\zeta$,  such that,  for all $t \in \R$,  $$\oA_n(t) \geq \widetilde \oB_n(t) - C_2 / \sqrt n.$$
	\end{lemma}
	
	\begin{proof}
	Using  \eqref{eq:coeff-split} and the definition of $\oA_n(t)$,  it follows that
	$$	\oA_n(t) \geq \sqrt{n}  \, \mathbf E \Big(  \mathbf 1_{t+ \sigma(S_n,x) +\log d(S_n x, H_y) - n \gamma\in [a, b]}\mathbf 1_{\log d(S_n x, H_y)\geq -B\log n} \Big).$$
		Recall that  when $S_n x\in\supp(\chi_k)$, one has $-(k+1) \zeta \leq \log d(S_n x,H_y) \leq - (k-1) \zeta$,  so
	$$  \mathbf 1_{t+ \sigma(S_n,x) +\log d(S_n x, H_y) - n \gamma\in [a, b]}\geq \mathbf 1_{\sigma(S_n,x)  - n \gamma\in [a-t+(k+1)\zeta, b-t+(k-1)\zeta]} \geq \widetilde\psi_{t,k}\big(\sigma(S_n,x)  - n \gamma\big).      $$
	Using that $\mathbf 1_{\log d(w, H_y)\geq -B\log n}  \geq \sum_{0\leq k\leq B\zeta^{-1}\log n - 1} \chi_k(w)$, it follows that
	\begin{align*}
	\mathbf 1_{t+ \sigma(S_n,x) +\log d(S_n x, H_y) - n \gamma\in [a, b]} \mathbf 1_{\log d(S_n x, H_y)\geq -B\log n} 
	\geq   \sum_{0\leq k\leq B\zeta^{-1}\log n-1} \widetilde \psi_{t,k}\big(\sigma(S_n,x)-n\gamma\big)\chi_k(S_n x)  . 
	\end{align*}
	
	Therefore, 	if  $k_0:= \lfloor B \zeta^{-1} \log n \rfloor$,  then
	\begin{equation*}
	\oA_n(t) \geq	 \sum_{0\leq k\leq B\zeta^{-1}\log n-1} \oE_n\big(\widetilde\psi_{t,k}\cdot\chi_k  \big) = \widetilde \oB_n(t) -   \oE_n\big(\widetilde \psi_{t,k_0}\cdot\chi_{k_0}  \big) -   \oE_n\big(\widetilde \psi_{t,0}\cdot \Phi_n^\star \big).
	\end{equation*}
	
	Arguing as in the proof of Lemma \ref{lemma:llt-ineq-1}, we see that the last two terms above are $\gtrsim - 1 / \sqrt n$. The lemma follows.
 	\end{proof}		
	
		Let $0<\delta<1$. By Lemma \ref{lemma:conv-fourier-approx},  there exists a smooth function $\widetilde \psi^-_{\delta}$  such that $\widehat {\widetilde \psi^-_{\delta}}$ has support in $[-\delta^{-2},\delta^{-2}]$,  $$\widetilde \psi^-_{\delta} \leq \widetilde\psi,\quad \lim_{\delta\to 0} \widetilde \psi^-_{\delta} =\widetilde  \psi    \quad \text{and} \quad  \lim_{\delta\to 0} \big \|\widetilde \psi^-_{\delta} -\widetilde\psi \big \|_{L^1} = 0.$$ 
	Moreover, $\norm{\widetilde \psi_{\delta}^-}_\infty$, $\norm{ \widetilde\psi_{\delta}^-}_{L^1}$ and $\|\widehat{ \widetilde \psi^-_{\delta}}\|_{\Cc^1}$ are bounded by a constant independent of $\delta$ and $\zeta$ since the support of $\psi$ is contained in $[a,b]$. We warn that, even if $\widetilde \psi$ is non-negative,   $\widetilde \psi^-_{\delta}$ might take negative values.
	
  	For $t\in\R$ and $k\in\N$, consider the translations
	$$ \widetilde\psi_{t,k}^-(u):= \widetilde \psi_{\delta}^-(u+t-k\zeta)$$ and define
	$$\widetilde \oR_n(t):= \sum_{0\leq k\leq B\zeta^{-1}\log n} \oE_n\big(\widetilde \psi_{t,k}^-\cdot\chi_k  \big)+  \oE_n\big(\widetilde \psi_{t,0}^-\cdot \Phi_n^\star \big)  $$ and  $$\widetilde \oS_n (t):={1\over 2\pi}\widehat{\widetilde  \psi_{\delta}^-}(0)  \int_{-\infty}^\infty e^{i\xi{t\over \sqrt n}} e^{-{ \varrho^2\xi^2 \over 2}}  \,\diff \xi={1\over \sqrt{2\pi} \, \varrho} e^{-{t^2\over2\varrho^2 n}}\int_{\R} \widetilde  \psi_{\delta}^- (u)\,\diff u.$$

	\begin{lemma} \label{lemma-R-S-2}
		Fix $0< \delta < 1$ and $0<\zeta \leq 1$ small enough. Then, there exists a constant $\widetilde C_{\zeta,\delta}>0$ such that, for all $n \geq 1$,
	$$\sup_{t\in\R}\big| \widetilde\oR_n(t)- \widetilde \oS_n(t) \big| \leq {\widetilde C_{\zeta,\delta}\over \sqrt[3] n}.$$
	\end{lemma}
	
	\begin{proof}
	By the same computations as the ones from Subsection \ref{subsec:LLT-upper-limit}, we obtain the identity
	$$\widetilde \oR_n(t)= {1\over 2\pi}\int_{-\delta^{-2} \sqrt n}^{\delta^{-2} \sqrt n} \widehat {\widetilde \psi^-_{\delta}} \Big({\xi\over \sqrt n}\Big) e^{i\xi{t\over \sqrt n}}\cdot e^{-i\sqrt n\gamma}\oP_{{i\xi\over \sqrt n}}^n(\Phi_{n,\xi} +\Phi_{n}^{\star}  )(x) \,\diff \xi,$$
	where $\Phi_{n,\xi}$ and $\Phi_{n}^{\star}$ are defined in  \eqref{eq:Phi-star-def}  and \eqref{eq:Phi-xi-def} respectively. The proof of Lemma \ref{lemma-R-S} can be repeated by using  ${\widetilde \psi^-_{\delta}}$ instead of $\psi_\delta^+$. This yields the desired estimate.
	\end{proof}
	
	We can now obtain the lower bound.

\begin{proof}[Proof of Proposition \ref{prop:LLT-lower-limit}]
Lemmas \ref{lemma:llt-ineq-2} and \ref{lemma-R-S-2} and the fact that $\widetilde \oB_n(t) \geq \widetilde \oR_n(t)$ give that
	$$\oA_n(t) \geq \widetilde \oS_n(t) - {\widetilde C_{\zeta,\delta}\over \sqrt[3] n} - {C_2 \over \sqrt n}  \quad \text{for all } \,\, t \in \R.$$
	Arguing as in the proof of Proposition \ref{prop:LLT-upper-limit} and recalling that $\int_{\R} \widetilde \psi (u)\,\diff u = b-a - 3\zeta$,  we get that, for every fixed $n$ and $\zeta$,  
	$$  \Big| \widetilde \oS_n(t)- e^{-{t^2\over2\varrho^2 n}}{b-a - 3\zeta\over \sqrt{2\pi} \, \varrho} \Big| \leq {1\over \sqrt{2\pi} \, \varrho}  \big\|   \widetilde \psi^-_{\delta} -\widetilde\psi \big \|_{L^1}.$$
	
	Therefore,
	$$ \oA_n(t) -   e^{-\frac{t^2}{2 \varrho^2 n}} {b-a\over \sqrt{2 \pi}\,\varrho} \geq  - e^{-{t^2\over2\varrho^2 n}} {3\zeta\over \sqrt{2\pi} \, \varrho} -   {1\over \sqrt{2\pi} \, \varrho}  \big\|   \widetilde \psi^-_{\delta} -\widetilde\psi \big \|_{L^1} - {\widetilde C_{\zeta,\delta}\over \sqrt[3] n} - {C_2 \over \sqrt n}, $$ and
	$$\liminf_{n\to \infty}  \inf_{t\in\R} \bigg( \oA_n(t) -   e^{-\frac{t^2}{2 \varrho^2 n}} {b-a\over \sqrt{2 \pi}\,\varrho} \bigg)  \geq  - {3\zeta\over \sqrt{2\pi} \, \varrho} -   {1\over \sqrt{2\pi} \, \varrho}  \big\|   \widetilde \psi^-_{\delta} -\widetilde\psi \big \|_{L^1}.$$
	Since $0<\delta<1$ and $0 < \zeta \leq 1$ are arbitrary and $ \big\|   \widetilde \psi^-_{\delta} -\widetilde\psi \big \|_{L^1}$ tends to zero as $\delta \to 0$, the proposition follows.
	
\end{proof}

Now, the proof of  Theorem \ref{thm:LLT-coeff} can be concluded.

\begin{proof}[Proof of Theorem \ref{thm:LLT-coeff}]
 Recall that the conclusion of Theorem \ref{thm:LLT-coeff} is equivalent to the limit \eqref{eq:LLT-main-limit}. Denote $\mathbf f_n(t):= \oA_n(t)  -   e^{-\frac{t^2}{2 \varrho^2 n}} {b-a\over \sqrt{2 \pi}\,\varrho}$. Propositions \ref{prop:LLT-upper-limit} and  \ref{prop:LLT-lower-limit} give that $$\limsup_{n\to \infty}\sup_{t\in\R} \mathbf f_n(t) \leq 0 \quad \text{and} \quad \liminf_{n\to \infty}\inf_{t\in\R} \mathbf f_n(t) \geq 0$$ respectively. This clearly implies that $\lim_{n\to \infty}\sup_{t\in\R} |\mathbf f_n(t)| = 0$, yielding  \eqref{eq:LLT-main-limit}.  It is clear  that all of our estimates are uniform in $x \in \P^{d-1}$ and $y\in (\P^{d-1})^*$.  The proof of the theorem is finished.
\end{proof}


\end{document}